\documentclass[11pt,a4 paper,oneside, reqno]{amsart}
\usepackage[utf8]{inputenc}
\usepackage[OT1]{fontenc}
\usepackage[english]{babel}
\usepackage{dsfont}
\usepackage{amsfonts}
\usepackage{amsmath}
\usepackage{bbm}
\usepackage{wasysym}
\usepackage{amsthm}
\usepackage{amssymb}
\usepackage{float}
\usepackage{textcomp}
\usepackage{xcolor}
\usepackage{graphicx}
\usepackage{fancybox}
\usepackage{fancyhdr}
\usepackage{mathrsfs}
\usepackage{tikz}
\usepackage{centernot}
\usepackage{listings}
\usepackage{faktor}
\usepackage{bm}
\usepackage{setspace}
\usepackage{hyperref}
\usepackage{newlfont}
\usepackage{geometry}
\usepackage{ccaption}
\usepackage{tipa}
\usepackage[autostyle,italian=guillemets]{csquotes}
\usepackage{comment}

\usepackage{guit}
\usepackage{tikz-cd}
\usepackage{enumerate}

\everymath{\displaystyle}

\theoremstyle{definition}
\newtheorem{Def}{Definition}[section]

\newtheorem{es}[Def]{Example}
\newtheorem{ese}[Def]{Examples}

\theoremstyle{remark}
\newtheorem{obs}[Def]{Remark}
\theoremstyle{plain}
\newtheorem{prop}[Def]{Proposition}

\newtheorem{lema}[Def]{Lemma}

\newtheorem{cor}[Def]{Corollary}

\newtheorem{teo}[Def]{Theorem}

\newcommand{\bo}{\mathbf}
\newcommand{\A}{{\mathcal A}}
\newcommand{\B}{{\mathcal B}}
\newcommand{\C}{{\mathcal C}}
\newcommand{\D}{{\mathcal D}}
\newcommand{\E}{{\mathcal E}}
\newcommand{\F}{{\mathcal F}}
\newcommand{\G}{{\mathcal G}}

\newcommand{\I}{{\mathcal I}}

\newcommand{\K}{{\mathcal K}}

\newcommand{\M}{{\mathcal M}}

\renewcommand{\P}{{\mathcal P}}

\renewcommand{\S}{{\mathcal S}}

\newcommand{\V}{{\mathcal V}}
\newcommand{\W}{{\mathcal W}}

\newcommand{\mt}{\mathscr}
\newcommand{\tx}{\textnormal}

\newcommand{\colim}{\operatornamewithlimits{colim}}

\newcommand{\todot}{%
	\mathrel{\ooalign{\hfil$\vcenter{
				\hbox{$\scriptscriptstyle\bullet$}}$\hfil\cr$\to$\cr}
	}%
}

\makeatletter
\newcommand{\changeoperator}[1]{%
	\csletcs{#1@saved}{#1@}%
	\csdef{#1@}{\changed@operator{#1}}%
}
\newcommand{\changed@operator}[1]{%
	\mathop{%
		\mathchoice{\textstyle\csuse{#1@saved}}
		{\csuse{#1@saved}}
		{\csuse{#1@saved}}
		{\csuse{#1@saved}}%
	}%
}
\makeatother

\makeatletter
\def\@tocline#1#2#3#4#5#6#7{\relax
	\ifnum #1>\c@tocdepth 
	\else
	\par \addpenalty\@secpenalty\addvspace{#2}%
	\begingroup \hyphenpenalty\@M
	\@ifempty{#4}{%
		\@tempdima\csname r@tocindent\number#1\endcsname\relax
	}{%
		\@tempdima#4\relax
	}%
	\parindent\z@ \leftskip#3\relax \advance\leftskip\@tempdima\relax
	\rightskip\@pnumwidth plus4em \parfillskip-\@pnumwidth
	#5\leavevmode\hskip-\@tempdima
	\ifcase #1
	\or\or \hskip 1em \or \hskip 2em \else \hskip 3em \fi%
	#6\nobreak\relax
	\hfill\hbox to\@pnumwidth{\@tocpagenum{#7}}\par
	\nobreak
	\endgroup
	\fi}
\makeatother

\DeclareUnicodeCharacter{221E}{$\infty$}

\changeoperator{sum}
\changeoperator{prod}
\changeoperator{coprod}

\title{Flat vs. filtered colimits in the enriched context}
\author{Stephen Lack and Giacomo Tendas}
\address{School of Mathematical and Physical Sciences, Macquarie University NSW 2109, 
	Australia}
\email{steve.lack@mq.edu.au}
\address{School of Mathematical and Physical Sciences, Macquarie University NSW 2109, 
	Australia}
\email{giacomo.tendas@mq.edu.au}
\date{\today}
\thanks{The first-named author acknowledges with gratitude the support of an Australian Research Council Discovery Project DP190102432. The second-named author gratefully acknowledges the support of an International Macquarie University Research Excellence Scholarship.}

\begin{document}
	
\begin{abstract}
	The importance of accessible categories has been widely recognized; they can be described as those freely generated in some precise sense by a small set of objects and, because of that, satisfy many good properties. More specifically {\em finitely accessible} categories can be characterized as: (a) free cocompletions of small categories under filtered colimits, and (b) categories of flat presheaves on some small category.\\
	The equivalence between (a) and (b) is what makes the theory so general and fruitful.
	
	Notions of enriched accessibility have also been considered in the literature for various bases of enrichment, such as $\bo{Ab},\bo{SSet},\bo{Cat}$ and $\bo{Met}$. The problem in this context is that the equivalence between (a) and (b) is no longer true in general. The aim of this paper is then to:\begin{enumerate}
		\item give sufficient conditions on $\V$ so that (a) $\Leftrightarrow$ (b) holds;
		\item give sufficient conditions on $\V$ so that (a) $\Leftrightarrow $ (b) holds up to Cauchy completion;
		\item explore some examples not covered by (1) or (2).
	\end{enumerate} 
\end{abstract}	
	
\maketitle
	
\tableofcontents

\section{Introduction}

The idea of flatness comes from homological algebra, but has since been incorporated into category theory in many contexts \cite{BS1983, Car2012, CV98:articolo, ObRo70:articolo}, perhaps most importantly in the theory of accessible categories \cite{AR94:libro,Lai81:articolo,MP89:libro}, which in turn builds on and extends the theory of locally presentable categories \cite{GU71:libro}. The theory of accessible categories involves aspects of category theory, universal algebra, logic, and model theory. It has also been heavily used in abstract homotopy theory, for example in the context of Smith's theorem \cite{Bek00} or Dugger's work on presentations for model categories \cite{DD01b,DD01a},  and also via its generalization to $\infty$-categories \cite{Lur09:libro}.

A functor $M\colon  \C^{op}\to\bo{Set}$ is called {\em flat} if its left Kan extension $\tx{Lan}_Y M \colon [\C,\bo{Set}]\to\bo{Set}$ along the Yoneda embedding preserves finite limits. Equivalently $ M $ is flat if and only if its category of elements is filtered; or even: $ M $ is flat if and only if it is a filtered colimit of representable functors. This outlines a deep connection between flatness and filtered colimits. 

This connection plays a key role in the theory of accessible categories: a category is finitely accessible by definition if it is the free cocompletion of a small category $\C$ under filtered colimits; by the observation above this is the same as saying that $\A$ is equivalent to the category $\tx{Flat}(\C^{op},\bo{Set})$ of flat presheaves on a small category $\C$. In particular, if $\A$ is finitely accessible, we can take $\C$ to be the full subcategory $\A_f$ of finitely presentable objects in $\A$. This is a fundamental step in the characterization of accessible categories as models of sketches and of first order theories.

The situation becomes rather more complicated when we move to enriched category theory. Let us fix a base of enrichment $\V=(\V_0,\otimes,I)$ which is symmetric monoidal closed and locally finitely presentable as a closed category \cite{Kel82:articolo}. In this setting a weighted notion of finite limit has been introduced by Kelly \cite{Kel82:articolo}; then conical filtered colimits commute in $\V$ with these finite weighted limits. However conical colimits are not generally enough when enrichment is involved; this means that there might be a wider class of weighted colimits which commute with finite weighted limits in $\V$. That is exactly where the notion of flat $\V$-functor comes into play:

\begin{Def}[\cite{Kel82:articolo}]
	We say that a $\V$-functor $ M \colon \C^{op}\to \V$ is {\em flat} if $\tx{Lan}_Y M \colon [\C,\V]\to\V$ preserves all finite weighted limits.
\end{Def}

Equivalently, $ M $ is flat if and only if $ M $-weighted colimits commute with finite limits in $\V$. If $\V=\bo{Ab}$ and $\C$ is a one object $\bo{Ab}$-category, then $R:=\C(*,*)\in\bo{Ab}$ is a ring and a $\V$-functor $ M \colon\C^{op}\to \bo{Ab}$ is just an $R$-module $M$, so that $[\C,\bo{Ab}]\cong R\tx{-Mod}$. Moreover $\tx{Lan}_YM\cong M\otimes-$; since right exact additive functors between abelian categories are left exact if and only if they preserve monomorphisms, we recover the algebraic notion of flatness introduced by Serre in \cite{Serre1956}.

We can now talk about flat-weighted colimits for $\V$-categories; these include, but do not reduce to, the conical filtered ones. For instance every absolute weighted colimit is flat but need not be filtered. An explicit example can be given for $\V=\bo{Ab}$: finite direct sums are absolute, and hence flat, but they are not filtered \cite[Example~9.2]{BQR98}. Therefore, depending on which class one decides to work with, two different notions of accessibility can be introduced. Historically, the first to be considered was based on flat weights.

\begin{Def}[\cite{BQ96:articolo}]
	A $\V$-category $\A$ is called {\em finitely accessible} if it is the free cocompletion of a small $\V$-category under flat-weighted colimits.
\end{Def}

On the bright side, this captures many of the characterization theorems from the ordinary setting. For instance a $\V$-category $\A$ is finitely accessible if and only if it is equivalent to $\tx{Flat}(\C^{op},\V)$, the full subcategory of $[\C^{op},\V]$ spanned by the flat $\V$-functors, for some small $\C$; while, moving to the infinitary setting, a $\V$-category is accessible if and only if it is the $\V$-category of models of a $\V$-sketch \cite[Corollary~7.9]{BQR98}. The problem with this notion, however, is that flat $\V$-functors can be hard to describe, and so it can be difficult to tell whether or not an enriched category is accessible in this sense. Perhaps for this reason, the theory remained relatively undeveloped for some time.

More recently, however, various authors have used accessibility in the enriched setting. An early example, involving the additive case $\V=\bo{Ab}$, was the work of Prest on the model theory of modules, as in \cite{Pre11:libro}. There followed various homotopical examples \cite{Bou2021:articolo, BLV:articolo,LR12:articolo}, involving $\V=\bo{Cat}$ and $\V=\bo{SSet}$ as bases of enrichment. Each of these cases was based on filtered colimits rather than flat ones, and implicitly or explicitly relied on the following notion.

\begin{Def}
	A $\V$-category $\A$ is {\em conically finitely accessible} if it is the free cocompletion of a small $\V$-category under conical filtered colimits.
\end{Def}

This is more straightforward to work with but lacks the connection with sketches which was after all the original motivation for the notion of accessibility.

Thus it is natural to ask what is the relationship between conically finitely accessible enriched categories and finitely accessible enriched categories. If they were in fact the same, or if at least we could understand well the relationship between them, then we would in some sense have the best of both worlds: both the relationship with sketches and other good theoretical properties, as well as the use of the simpler filtered colimits rather than flat-weighted ones.

Recall that an accessible category is complete if and only if it is cocomplete, in which case it is said to be locally presentable. Many of the trickiest aspects of accessible categories disappear under the assumption of completeness/cocompleteness, and this remains true in the enriched setting; in particular if $\A$ is a cocomplete or complete $\V$-category, then $\A$ is finitely accessible if and only if it is conically finitely accessible; this is essentially the content of \cite[Theorem~6.3]{BQR98}. One context in which the homotopical examples of enriched accessible categories referred to above might arise is when one ``should" have a locally presentable category, but in fact the limits in question are only homotopical ones. This is the case for a number of the examples considered in \cite{Bou2021:articolo, BLV:articolo,LR12:articolo}.
  
{\em The aim of the paper is precisely to address these problems by giving an explicit description of flat-weighted colimits for certain classes of base of enrichment, including most of the important examples of locally presentable bases. We use this to characterize accessible $\V$-categories in many cases.}

This is a 40-year-old problem: in \cite[Section~6.4]{Kel82:libro}, Kelly poses the question of whether, for any locally finitely presentable base, every flat presheaf is a filtered colimit of representables, and states his inability to prove this. As observed in \cite{BQR98} and above, this is actually false for the case $\V=\bo{Ab}$; but solving it in full generality is probably out of reach at this stage. The situation is analogous to the related hard problem of describing the absolute colimits over a given base: see \cite{Law73:cauchy,Nik17:articolo,NST2020cauchy,Str1983:articolo} for various instances of this.

In fact we give a positive answer to Kelly's question in a large number of examples, including the key cases $\V=\bo{Cat}$ and $\V=\bo{SSet}$ mentioned above. It follows that in these cases the two notions of accessibility agree. And in a still larger class of examples, we are nonetheless able to give a full characterization of flat $\V$-functors: of the examples of locally presentable base considered by Kelly in \cite{Kel82:libro}, the only one we are not able to cover is the category $\E:=\tx{Shv}(\S)$ of sheaves over a site $\S$ (that is, a Grothendieck topos), although we can cover it if the functor $\E(1,-)$ is weakly cocontinuous in the sense of Section~\ref{con=flat}, and so in particular if $1$ is connected, finitely presentable, and projective.

The paper \cite{Bou2021:articolo} contains some powerful techniques for proving that a wide range of 2-categories of categories with structure are conically accessible as $\bo{Cat}$-enriched categories; the structures in question should contain ``no equations between objects”, and the morphisms are functors which preserve the structure up to coherent isomorphism. One typical example is the 2-category of monoidal categories, strong monoidal functors, and monoidal natural transformations; another is the 2-category of regular categories, regular functors, and natural transformations. In \cite{BL21accessible}, it was shown how to adapt these techniques to the simplicially enriched case, and in particular to show that most of the key examples of $\infty$-cosmoi studied in \cite{riehl_verity_2022} are conically accessible as $\bo{SSet}$-enriched categories. By the results proved here, these conically accessible $\bo{Cat}$-enriched or $\bo{SSet}$-enriched categories are in fact accessible; thus they are also sketchable, and the whole theory of enriched accessible categories applies. This in turn allows us, for example, to consider models of the corresponding enriched sketches in other (suitable) enriched categories than $\V$, and deduce the accessibility of the resulting enriched categories of models.

Given any complete and cocomplete $\V$ and any small $\V$-category $\C$, we can form the underlying ordinary category $[\C^{op},\V]_0$ of the presheaf $\V$-category: this consists of the $\V$-enriched presheaves and $\V$-enriched natural transformations. We can also form the presheaf category $[\C^{op}_0,\bo{Set}]$ on the underlying ordinary category. There is an adjunction
\begin{center}
	
	\begin{tikzpicture}[baseline=(current  bounding  box.south), scale=2]

		\node (f) at (0,0.4) {$[\C^{op},\V]_0$};
		\node (g) at (1.7,0.4) {$[\C^{op}_0,\bo{Set}]$};
		\node (h) at (0.84, 0.45) {$\perp$};
		\path[font=\scriptsize]

		([yshift=-1.3pt]f.east) edge [->] node [below] {$\mt U$} ([yshift=-1.3pt]g.west)
		([yshift=2pt]f.east) edge [bend left,<-] node [above] {$\mt F$} ([yshift=2pt]g.west);
	\end{tikzpicture}
	
\end{center}
between these, induced by the underlying functor $Y_0\colon \C_0\to[\C^{op},\V]_0$ of the enriched Yoneda embedding: here $\mt U$ sends $ M \colon \C^{op}\to\V$ to $[C^{op},\V]_0(Y_0-, M )\cong \V_0(I,M_0-)$ and $\mt F$ sends $ N $ to the colimit $ N *Y_0$ of $Y_0$ weighted by $ N $ (this can also be seen as the colimit of $\tx{El}( N )\stackrel{\pi}{\longrightarrow}\C_0\stackrel{Y_0}{\longrightarrow}[\C^{op},\V]_0$). 

It is true in general that $\mt F$ sends a flat (in the ordinary sense) presheaf on $\C_0$ to a flat (in the enriched sense) presheaf on $\C$, essentially because filtered colimits of flat presheaves are flat. It is not necessarily true that $\mt U$ preserves flatness, but it is so in many examples, as we shall see. Our basic strategy will be to show in particular cases, sometimes under further assumptions on $\C$, that\begin{enumerate}
	\item[(I)] $\mt U$ does preserve flatness;
	\item[(II)] if $ M $ is a flat presheaf on $\C$, then the component $\epsilon_ M \colon \mt F\mt U M \to M $ of the counit is invertible. 
\end{enumerate}

Given (II), any flat presheaf $ M $ on $\C$ is an $\mt F\mt U M $-weighted colimit of representables. Given (I), this colimit can be calculated as a filtered colimit. Combining these, it follows that the flat presheaves on $\C$ are the closure of the representables under filtered colimits. 

In Section~\ref{con=flat}, we give conditions on $\V$ under which this is the case for {\em all} small $\C$, and deduce that for such $\V$, the existence and preservation of flat weighted colimits is equivalent to that of filtered colimits (Theorem~\ref{flat-preservation}). As a consequence the notions of $\alpha$-accessibility and conical $\alpha$-accessibility agree (Theorem~\ref{2-cataccessibility}). 

Examples of $\V$ satisfying these conditions include the cartesian closed categories $\bo{Set}$ of sets, $\bo{Cat}$ of small categories, $\bo{SSet}$ of simplicial sets, $\mathbbm{2}$ of the free-living arrow, $\bo{Pos}$ of partially ordered sets, and many others.

In Section~\ref{flat=absolute+filtered}, we give conditions on $\V$ under which this (I) and (II) hold provided that $\C$ has certain $\V$-enriched absolute colimits (Proposition~\ref{flat+absolute=filtered}). It then follows easily (Theorem \ref{acc=conacc+cauchy}) that a $\V$-category is $\alpha$-accessible if and only if it is conically $\alpha$-accessible and has these absolute colimits. (An $\alpha$-accessible category always has absolute colimits.) For the $\V$ which we study in Section~\ref{flat=absolute+filtered}, the absolute colimits in question are finite direct sums and copowers by dualizable objects. 

Examples of $\V$ satisfying these conditions include the monoidal categories $\bo{CMon}$ of commutative monoids, $\bo{Ab}$ of abelian groups, $R\tx{-}\bo{Mod}$ of $R$-modules for a commutative ring $R$, and $\bo{GAb}$ of graded abelian groups, and more generally $\bo{G}\tx{-}\bo{Gr}(R\tx{-}\bo{Mod})$ of $\bo{G}$-graded $R$-modules for an abelian group $\bo{G}$ and a commutative ring $R$.

In Section~\ref{counterexample}, we investigate the case where $\V$ is the cartesian closed category $\bo{Set}^G$ of $G$-sets, for a non-trivial finite group $G$, and show that in this case $\alpha$-accessibility is strictly stronger than conical $\alpha$-accessibility and the existence of absolute colimits (Corollary \ref{acc-not-conacc+absolute}). 

In Section~\ref{DG}, we investigate one further class of examples related to those in Section~\ref{flat=absolute+filtered}, and including for example $\V=\bo{DGAb}$. In this case (II) still holds when $\C$ has some finite direct sums and copowers by dualizable objects, while (I) doesn't seem to be true even with this further assumption. What is true is that, when $\C$ has those absolute colimits and $M$ is flat, the ordinary functor $\mt U M $ will be part of what we call a protofiltered diagram. Then we prove that flat colimits are generated by the absolute ones plus these protofiltered colimits (Theorem~\ref{flat=protofilt+abs}).

\section{Background notions}

From now on $\V=(\V_0,\otimes,I)$ is taken to be a symmetric monoidal closed category which is moreover locally $\alpha$-presentable as a closed category \cite{Kel82:articolo} for a regular cardinal $\alpha$, meaning that $I$ is $\alpha$-presentable and whenever $X,Y$ are $\alpha$-presentable then $X\otimes Y$ is so.

We follow the notations of \cite{Kel82:libro}, with the only change that ``indexed'' colimits are here called ``weighted'', as is now standard. In particular, given a $\V$-category $\A$ we denote by $\A_0$ its underlying ordinary category; similarly if $F\colon \A\to\B$ is a $\V$-functor we denote by $F_0\colon \A_0\to\B_0$ the corresponding ordinary functor underlying $F$. For any ordinary category $\K$ we denote by $\K_\V$ the free $\V$-category over $\K$. Our $\V$-categories are allowed to have a large set of objects, unless specified otherwise.

Concerning weighted limits and colimits, a {\em weight} is a $\V$-functor $ M \colon \C^{op}\to\V$ where $\C$ is a small $\V$-category. Given such a weight $ M $ and a $\V$-functor $H\colon \C\to\A$, we denote by $ M *H$ (if it exists) the colimit of $H$ weighted by $ M $; dually weighted limits are denoted by $\{ N ,K\}$ for $ N \colon \C\to\V$ and $K\colon \C\to\A$. When $\C$ is the unit $\V$-category $\I=\{*\}$ the weighted colimit is called the {\em copower} of $A\in\A$ by $X\in \V$ and is denoted by $X\cdot A$ ; dually we denote powers by $X\pitchfork A$. Conical limits and colimits are special cases of weighted ones; they coincide with those weighted by $\Delta I\colon \B_\V^{op}\to\V$ for some ordinary category $\B$. The conical colimit of a $\V$-functor $T_{\V}\colon \B_\V\to\A$, if it exists, will also be the ordinary colimit of the transpose $T\colon \B\to\A_0$ in $\A_0$. Conversely, the conical colimit of $T_\V$ exists when the ordinary limit of $T$ exists and is preserved by each representable $\A(-,C)_0\colon \A^{op}_0\to\V_0$. This latter preservation condition is automatic if $\A$ has powers by all objects in a strong generator for $\V_0$, but not in general.

\begin{Def}[\cite{Kel82:articolo}]
	We say that a weight $ M \colon \C^{op}\to\V$ is {\em $\alpha$-small} if $\C$ has less than $\alpha$ objects, $\C(C,D)\in\V_{\alpha}$ for any $C,D\in\C$, and $ M (C)\in\V_\alpha$ for any $C\in\C$. An $\alpha$-small (weighted) limit is one taken along an $\alpha$-small weight. We say that a $\V$-category $\C$ is $\alpha$-complete if it has all $\alpha$-small limits; we say that a $\V$-functor is $\alpha$-continuous if it preserves all $\alpha$-small limits.
\end{Def}

Both conical $\alpha$-small limits and powers by $\alpha$-presentable objects are examples of $\alpha$-small limits and together are enough to generate all $\alpha$-small weighted limits \cite[Section~4]{Kel82:articolo}. 

\begin{Def}[\cite{Kel82:articolo}]
	A weight $ M \colon\C^{op}\to\V$ is called $\alpha$-flat if its left Kan extension $\tx{Lan}_Y M \colon[\C,\V]\to\V$ along the Yoneda embedding is $\alpha$-continuous. An $\alpha$-flat colimit is one taken along an $\alpha$-flat weight.
\end{Def}

Equivalently, a weight $ M $ is $\alpha$-flat if $ M $-weighted colimits commute with $\alpha$-small limits in $\V$ (this is because $\tx{Lan}_Y M \cong  M *-$). A weight which is $\alpha$-flat for any $\alpha$ is called {\em Cauchy}; these can be described as those $ M \colon\C^{op}\to\V$ for which $ M *-$ is continuous \cite{Str1983:articolo}.

The following summarizes a few properties of $\alpha$-flat $\V$-functors:

\begin{prop}[\cite{Kel82:articolo}]$ $
	\begin{enumerate}\setlength\itemsep{0.25em}
		\item The $\alpha$-flat $\V$-functors are closed in $[\C^{op},\V]$ under $\alpha$-flat colimits.
		\item Conical $\alpha$-filtered colimits are $\alpha$-flat.
		\item Cauchy (weighted) colimits are $\alpha$-flat.
	\end{enumerate}
\end{prop}

In the result below, as well as in the following sections, we consider the category of elements associated to an ordinary functor $F\colon \C^{op}\to\bo{Set}$. Since there are two (dual) notions of such in the literature, we recall the definition below. 

\begin{Def}
	Let $F\colon \C^{op}\to\bo{Set}$ be an ordinary functor. The category of elements $\tx{El}(F)$ of $F$ can be described as follows:\begin{itemize}\setlength\itemsep{0.25em}
		\item an object is a pair $(C\in\C,x\in FC)$;
		\item a morphism $f\colon (C,x)\to (D,y)$ is an arrow $f\colon C\to D$ in $\C$ for which $F(f)(y)=x$.
		\end{itemize}
	Identities and composition are as in $\C$. We denote by $\pi\colon \tx{El}(F)\to\C$ the projection.
\end{Def}

Note that we will always take the categories of element of a contravariant functor as above.

\begin{Def}
	For any $\V$-functor $H\colon \C\to\V$, we define 
	$$H_I:=\V_0(I,H_0-)\colon \C_0\to\bo{Set}.$$ 
\end{Def}

This notation just introduced is used below and will occur also later on in the paper.

\begin{prop}[\cite{Kel82:articolo}]
	Let $ M \colon \C^{op}\to\V$ be a weight; the following are equivalent:\begin{enumerate}\setlength\itemsep{0.25em}
		\item $ M $ is $\alpha$-flat;
		\item $ M $ is an $\alpha$-flat colimit of representables.
	\end{enumerate}
	If $\C$ is $\alpha$-cocomplete they are further equivalent to: \begin{enumerate}\setlength\itemsep{0.25em}
		\item[(3)] $ M $ is $\alpha$-continuous;
		\item[(4)] $ M $ is a conical $\alpha$-filtered colimit of representables. 
	\end{enumerate}
	In that case the following isomorphism holds
	$$  M \cong\tx{colim} (\tx{El}( M _I)_{\V} \stackrel{\pi_\V}{\longrightarrow} \C \stackrel{Y}{\longrightarrow} [\C^{op},\V])$$
	where $\tx{El}( M _I)$ is $\alpha$-filtered.
\end{prop}

Point $(1)$ of the following result is \cite[Corollary~4.3]{BQR98}, but we do not know of a reference for $(2)$. 

\begin{lema}\label{flat-restriction}
	Let $J\colon \B\to\C$ be a $\V$-functor, and $ M \colon \B^{op}\to\V$ a weight; then:\begin{enumerate}\setlength\itemsep{0.25em}
		\item if $ M $ is $\alpha$-flat then $\tx{Lan}_{J^{op}} M$ is;
		\item if $J$ is fully faithful and $\tx{Lan}_{J^{op}} M $ is $\alpha$-flat then $ M $ is $\alpha$-flat as well.
	\end{enumerate}
\end{lema}
\begin{proof}
	By definition a functor $ M $ is $\alpha$-flat if its left Kan extension along the Yoneda embedding, which is the functor $ M *-\colon [\B,\V]\to\V$, is $\alpha$-continuous. 
	Note that the triangle below commutes;
	\begin{center}
		\begin{tikzpicture}[baseline=(current  bounding  box.south), scale=2]
			
			\node (a) at (0.7,0.6) {$[\B,\V]$};
			\node (c) at (0, 0) {$[\C,\V]$};
			\node (d) at (1.6, 0) {$\V$};
			
			\path[font=\scriptsize]
			
			(c) edge [->] node [above] {$[J,\V]\ \ \ \ \ \ \ \ $} (a)
			(a) edge [->] node [above] {$\ \ \ \ \ \  M *-$} (d)
			(c) edge [->] node [below] {$(\tx{Lan}_{J^{op}}M)*-$} (d);
			
		\end{tikzpicture}	
	\end{center} 
	indeed $(\tx{Lan}_{J^{op}}M)*F\cong  M *FJ$ by \cite[4.19]{Kel82:libro}. As a consequence if $ M $ is $\alpha$-flat then so is $\tx{Lan}_{J^{op}} M$ since $[J,\V]$ is continuous.
	Conversely, if $J$ is fully faithful and $\tx{Lan}_{J^{op}}M $ is $\alpha$-flat, then 
	\begin{equation*}
		\begin{split}
			 M *-&\cong ( M *-)\circ id_{[\B,\V]} \\
			&\cong ( M *-)\circ [J,\V]\circ \tx{Ran}_{J}\\
			&\cong (\tx{Lan}_{J^{op}}M*-)\circ \tx{Ran}_{J}\\
		\end{split}
	\end{equation*}
	where $id_{[\B,\V]}\cong[J,\V]\circ \tx{Ran}_{J}$ since $J$ is fully faithful. It follows that $ M *-$ is $\alpha$-continuous because $\tx{Lan}_{J^{op}}M*-$ is and $\tx{Ran}_{J}$ is continuous.
\end{proof}

\begin{obs}\label{final-filtered-subcat}
	In more familiar terms, this generalizes an easy-to-check fact about filtered categories:\begin{itemize}\setlength\itemsep{0.25em}
		\item if $J\colon \B\to\A$ is final and $\B$ is filtered, then $\A$ is filtered as well;
		\item if $J\colon \B\to\A$ is fully faithful and final, and $\A$ is filtered, then $\B$ is filtered as well.
	\end{itemize} 
	In the second point, we cannot drop the assumption that $J$ is fully faithful as the following example shows. Take the inclusion of the free-living pair into the free-living split pair; then the codomain is filtered (it is actually absolute) and the inclusion is final, but coequalizers are not filtered colimits.
\end{obs}

Next we recall the two notions of accessibility considered in the literature. The first is a straightforward generalization of the notion considered for $\V=\bo{Set}$: 

\begin{Def}
	Let $\alpha$ be a regular cardinal and $\A$ a $\V$-category with conical $\alpha$-filtered colimits; an object $A$ of $\A$ is called {\em conically $\alpha$-presentable} if $\A(A,-)\colon \A\to\V$ preserves conical $\alpha$-filtered colimits; we write $\A_\alpha^c$ for the full subcategory of $\A$ spanned by the conically $\alpha$-presentable objects. 
\end{Def}	

\begin{Def}	
	We say that $\A$ is {\em conically $\alpha$-accessible} if it has (conical) $\alpha$-filtered colimits and there exists a small $\C\subseteq\A_\alpha^c$ such that every object of $\A$ can be written as an $\alpha$-filtered colimit of objects from $\C$.
\end{Def}

Equivalently: a $\V$-category is conically $\alpha$-accessible if it is the free cocompletion of a small $\V$-category under $\alpha$-filtered colimits.

With this notion we are not able to recover the usual characterizations theorem from the ordinary case (such as the characterization of accessible categories as models of sketches). This can be solved if $\alpha$-filtered colimits are replaced by the $\alpha$-flat ones:

\begin{Def}{\cite{BQR98}}
	Let $\alpha$ be a regular cardinal and $\A$ a $\V$-category with $\alpha$-flat colimits; an object $A$ of $\A$ is called {\em $\alpha$-presentable} if $\A(A,-)\colon\A\to\V$ preserves $\alpha$-flat colimits; we write $\A_\alpha$ for the full subcategory of $\A$ spanned by the $\alpha$-presentable objects. 
\end{Def}	

\begin{Def}{\cite{BQR98}}	
	We say that $\A$ is {\em $\alpha$-accessible} if it has $\alpha$-flat colimit and there exists a small $\C\subseteq\A_\alpha$ such that every object of $\A$ can be written as an $\alpha$-flat colimit of objects from $\C$.
\end{Def}

Equivalently: a $\V$-category is $\alpha$-accessible if it is the free cocompletion of a small $\V$-category under $\alpha$-flat colimits.

These two notions of accessibility are not equivalent: conical accessibility doesn't in general imply accessibility; the reason being simply that having $\alpha$-flat colimits is stronger than having $\alpha$-filtered ones. On the other hand it is not known whether every $\alpha$-accessible $\V$-category is conically $\alpha$-accessible; this would require an explicit characterization of $\alpha$-flat colimits in terms of the $\alpha$-filtered ones. What is known is that every $\alpha$-accessible $\V$-category $\A$ is conically $\gamma$-accessible for some regular cardinal $\gamma$ greater than $\alpha$ (it's enough to consider $\gamma$ as in \cite[Theorem~7.10]{BQR98} and use points $(2)$ and $(3)$ of that Theorem~to show that $\A$ is conically $\gamma$-accessible).

\section{When flat equals filtered}\label{con=flat}

In this section we give sufficient conditions on the base $\V$ for $\alpha$-flat colimits to reduce to the usual $\alpha$-filtered ones. These conditions hold in many of the most important examples of locally presentable bases of enrichment, as Example~\ref{weaklstrongexamples} below illustrates. For most but not all of the bases which appear in Example~\ref{weaklstrongexamples}, the canonical functor $\V_0(I,-):\V_0\to\bo{Set}$ is cocontinuous and strong monoidal, which in turn easily implies our conditions. One example for which this is not the case, but for which our sufficient conditions still hold, is the category of pointed sets, equipped as usual with the smash product. 

Throughout this section we assume that $\V$ is locally $\alpha$-presentable as a closed category and that the unit $I$ satisfies the following conditions:\begin{enumerate}\setlength\itemsep{0.25em}
	\item[(a)] $\V_0(I,-)\colon \V_0\to\bo{Set}$ is {\em weakly cocontinuous}: for any diagram $H\colon \C\to\V_0$ the comparison map $\colim\V_0(I,H-)\twoheadrightarrow\V_0(I,\colim H)$ is a surjection.
	\item[(b)] $\V_0(I,-)\colon \V_0\to\bo{Set}$ is {\em weakly strong monoidal}: the comparison maps for the tensor product are surjections; in other words for each $X,Y\in\V_0$ the function $\V_0(I,X)\times\V_0(I,Y)\to \V_0(I,X\otimes Y)$, sending $(x,y)$ to $x\otimes y$, is a surjection.
\end{enumerate}

\begin{obs}
	Note that condition $(a)$ is equivalent to the fact that $\V_0(I,-)$ preserves cocones which are jointly a regular epimorphism; or even to the request that $I$ is regular projective and $\V_0(I,-)$ weakly preserve coproducts. Therefore $(a)$ is certainly satisfied whenever $\V_0(I,-)$ preserves coproducts and regular epimorphisms, and in particular when it is cocontinuous (most of our examples). Condition $(b)$ comes for free when $\V_0(I,-)$ is strong monoidal (most of our examples), and in particular when $\V_0$ is endowed with the cartesian closed structure. 
\end{obs}

\begin{obs}
	In condition $(b)$ it might be natural to ask for the map $1_I\colon 1\to\V_0(I,I)$ to be an epimorphism as well (and hence bijective), but it is not required for the results below. Nonetheless, this condition is satisfied for almost all our examples.
\end{obs}

\begin{ese}\label{weaklstrongexamples}
	Here is a list of examples of such bases of enrichment. In the following group each base $\V$ is endowed with the cartesian structure and $\V_0(I,-)$ is cocontinuous: \begin{enumerate}\setlength\itemsep{0.25em}
		\item $(\bo{Set},\times,1)$ for ordinary categories.
		\item $(\bo{Cat},\times,1)$ for 2-categories.
		\item $(\bo{SSet},\times,1)$ for simplicial categories.
		\item $(\mathbbm{2},\times,1)$ for posets. 
		\item The categories $\bo{Gpd}$ of groupoids, $\bo{Ord}$ of total orders, $\bo{Pos}$ of posets, and $\bo{rGra}$ of reflexive graphs with their cartesian closed structures.
		\item Any presheaf category $([\C^{op},\bo{Set}],\times, 1)$ for which $\C$ has a terminal object: the unit $1$ is representable in $[\C^{op},\bo{Set}]$ and hence homming out of it preserves all colimits.
		\item $(2\tx{-}\bo{Cat}_Q,\times, 1)$ the cartesian closed category of algebraically cofibrant 2-categories, see \cite{Cam21:articolo}. This base is locally presentable and the forgetful $V\colon 2\tx{-}\bo{Cat}_Q\to 2\tx{-}\bo{Cat}$ is a faithful left adjoint and is full on morphisms out of the terminal object, so that $2\tx{-}\bo{Cat}_Q(1,-)\cong 2\tx{-}\bo{Cat}(1,V-)$ is cocontinuous.
	\end{enumerate}
	The following are examples of bases for which $\V_0(I,-)$ is cocontinuous and strong monoidal but the monoidal structure is not (necessarily) cartesian: \begin{enumerate}\setlength\itemsep{0.25em}
		\item[(8)] $(\V\tx{-}\bo{Cat},\otimes,\I)$ with the tensor product inherited from $\V$, whenever $\V_0$ is locally presentable: $\V\tx{-}\bo{Cat}$ is locally presentable by \cite{KL2001:articolo}, and $\V\tx{-}\bo{Cat}(\I,-)\cong \tx{Ob}(-)$ is the functor that takes the underlying objects of a category; therefore it is cocontinuous and strong monoidal ($\tx{Ob}(\A\otimes \B)=\tx{Ob}(\A)\times\tx{Ob}(\B)$).
		\item [(9)] $(2\tx{-}\bo{Cat},\square ,1)$ with the ``funny tensor product'' \cite[Section~2]{Str96categorical}: same reasons as above.
		\item [(10)] $(2\tx{-}\bo{Cat},\boxtimes,1)$ with the pseudo Gray tensor product: same reasons as above (see \cite[Section~6]{Str96categorical}).
		\item [(11)] $(\bo{Met},\otimes,1)$ of generalized metric spaces (see \cite{AR20:articolo} and \cite{LR17:articolo}).
	\end{enumerate}
	In the next example the monoidal structure is not cartesian, the unit is not the terminal object, and $\V_0(I,-)$ is only weakly cocontinuous and weakly strong monoidal:
	\begin{enumerate}\setlength\itemsep{0.25em}
		\item[(12)] $(\bo{Set}_*,\wedge,I)$ the category of pointed sets endowed with the smash product. This is locally presentable being the co-slice  $1/\bo{Set}$. Since the unit is given by the pointed set $I=(\{0,1\},0)$, it follows that the functor $U:=\bo{Set}_*(I,-)\colon \bo{Set}_*\to\bo{Set}$ is just the underlying set functor. Note now that the tensor product $(A,a)\wedge (B,b)$ is defined as a quotient of the set $A\times B$; thus $U$ is weakly strong monoidal but not strong monoidal, nor does it preserve the unit. Moreover, it's easy to see that epimorphisms in $\bo{Set}_*$ are just surjections, and that the coproduct of a family $(A_i,a_i)_{i\in I}$ is the quotient of $\textstyle\sum_{i\in I} A_i$ obtained identifying all he $a_i$'s. As a consequence the functor $U$ preserves all regular epimorphisms and weakly preserves all coproducts, but doesn't preserve coproducts in the usual sense.
	\end{enumerate}
\end{ese}

\begin{obs}
	The last example can be generalized as follows. Assume that $\V=(\V_0,\otimes,I)$ satisfies the conditions $(a)$ and $(b)$ above. Then $1/\V_0$ is still locally presentable and symmetric monoidal closed with monoidal structure given by the smash product $\wedge$ induced by the tensor product on $\V$ (see for example \cite[Lemma~4.20]{EM2009permutative}); moreover the forgetful functor $U\colon 1/\V_0\to\V_0$ is monoidal and has a strong monoidal left adjoint $F$. Note now that, since the unit of $1/\V$ is $FI$, the functor $(1/\V_0)(FI,-)\cong \V_0(I,U-)$ weakly preserves the same colimits that both $U$ and $\V_0(I,-)$  weakly preserves. Since the comparison map $UX\otimes UY\to U(X\wedge Y)$ is a regular epimorphism in $\V_0$ (by construction) and $U$ is weakly cocontinuous, it follows that $(1/\V_0)(FI,-)$ is weakly cocontinuous and weakly strong monoidal as well. In conclusion $1/\V=(1/\V_0,\wedge,FI)$ still satisfies the conditions $(a)$ and $(b)$.
\end{obs}

The following is an easy consequence of the two conditions above, but is also the foundation of the results of this section.

\begin{lema}\label{goodunit}
	Let $ M \colon \C^{op}\to\V$ and $H\colon \C\to\V$ be two $\V$-functors. Then for each arrow $x\colon I\to  M *H$ there exist $C\in \C$, $y\colon I\to M  C$, and $z\colon I\to HC$ for which the triangle
	\begin{center}
		\begin{tikzpicture}[baseline=(current  bounding  box.south), scale=2]

			\node (b0) at (1.3,0.7) {$ M  C\otimes HC$};
			\node (c0) at (0,0.7) {$I$};
			\node (d0) at (1.3,0) {$ M *H$};
			
			\path[font=\scriptsize]
			
			(c0) edge [dashed, ->] node [above] {$ y\otimes z$} (b0)
			(b0) edge [->] node [right] {$\rho_C$} (d0)
			(c0) edge [->] node [below] {$x\ $} (d0);
			
		\end{tikzpicture}	
	\end{center}
	commutes; where the vertical map is taken from the colimiting cocone defining $ M *H$.
\end{lema}
\begin{proof}
	The counit of a colimit $ M * H$ determines a family
	$$ (\rho_C\colon  M  C\otimes HC \longrightarrow  M *H)_{C\in\C}$$
	which is jointly a regular epimorphism in $\V_0$. Since $\V_0(I,-)$ preserves such families it follows that $x$ factors through $\rho_C$ via a map $h\colon I\to  M  C\otimes HC$, for some $C\in\C$. Finally, by condition $(2)$ on the unit, $h=y\otimes z$ for some $y\colon I\to M  C$ and $z\colon I\to HC$. The claim then follows.
\end{proof}

\begin{obs}
	Given a pair $ M \colon \C^{op}\to\V$ and $H\colon \C\to\V$, we can consider the weighted colimit $ M *H$ in $\V$ and the ordinary weighted colimit $ M _I*H_I$ in $\bo{Set}$. Note that there is always a comparison map
	$$ c\colon  M _I*H_I\longrightarrow \V_0(I, M *H).$$
	It's easy to see that if conditions $(1)$ and $(2)$ hold then the map $c$ is a surjection (use the Lemma~above); thus in a certain sense the functor $\V_0(I,-)$ weakly preserves weighted colimits --- this can be made precise using change of base for composition of profunctors.\\
	When $\V_0(I,-)$ is moreover cocontinuous and strong monoidal the comparison map $c$ is actually an isomorphism: $\V_0(I, M *H)\cong  M _I*H_I$ (this is a general fact about cocontinuous functors and strong monoidal change of base).
\end{obs}

The following is a generalization of \cite[Proposition~6.6]{Kel82:articolo} to our context; the proof is very similar to that and is based on the lemma above.

\begin{cor}\label{filtered-elements}
	Let $ M \colon \C^{op}\to\V$ be a weight for which $ M *-\colon [\C,\V]\to \V$ preserves $\alpha$-small conical limits of representables and let $ M _I:=\V_0(I, M _0-)\colon \C^{op}_0\to\bo{Set}$; then the ordinary category $\tx{El}( M _I)$ is $\alpha$-filtered, and so $M_I$ is $\alpha$-flat.
\end{cor}
\begin{proof}
	Let $\B$ be an $\alpha$-small category and $H\colon \B\to \tx{El}( M _I)$ be a functor; we need to prove that $H$ has a cocone in $\tx{El}( M _I)$. Denote by $\pi\colon \tx{El}( M _I)\to \C_0$ the projection; then $H$ induces a cone $(I\stackrel{Hb}{\longrightarrow} M (\pi Hb))_{b\in\B}$, which is the same as an arrow 
	$$x\colon I\longrightarrow\lim_{b\in\B} M (\pi Hb).$$ 
	Now, since $ M *-$ preserves $\alpha$-small limits of representables, we obtain the following isomorphisms
	$$ \lim_{b\in\B} M (\pi Hb)\cong \lim_{b\in\B} ( M *\C(\pi Hb,-))\cong  M *(\lim_{b\in\B} \C(\pi Hb,-)). $$
	Therefore, by the previous Lemma, there exist $C\in \C$, $y\colon I\to M  C$, and $z\colon I\to \lim \C(\pi H-,C)$ which map down to $x$ when taking the colimit. Note now that to give $z$ is the same as to give a cone $\Delta I\to \C(\pi H-,C)$, which then corresponds to a cocone $(\eta_b\colon \pi Hb\to C)_{b\in\B}$ in $\C$. Finally the fact that $y\otimes z$ gets mapped down to $x$ means that $\eta$ is actually a cocone $(\eta_b\colon Hb\to (C,y))_{b\in\B}$ in $\tx{El}( M _I)$.
\end{proof}

This shows that in the current setting condition (I) from the Introduction holds. Next we turn to (II) --- such an $ M $ is actually an $\alpha$-filtered colimit of representables --- but for that we need some work. 

\begin{Def}\label{genequalizer}
	Let $h\colon Y\to X$ be a morphism in $\V$. Denote by $\mathbbm{2}_Y$ the $\V$-category with two objects $*_1$ and $*_2$, and with hom-objects $\mathbbm{2}_Y(*_i,*_i)=I$, $\mathbbm{2}_Y(*_2,*_1)=0$, and $\mathbbm{2}_Y(*_1,*_2)=Y$. Let $ N _h\colon \mathbbm{2}_Y\to\V$ be the weight for which $ N _h(*_1)=I$, $ N _h(*_2)=X$ and determined on the hom-objects by $( N _h)_{*_1,*_2}=h$.\\
	When $h$ is the co-diagonal $\nabla\colon X+X\to X$ we write $ N _X$ for $ N _\nabla$.
\end{Def}

Note that, when $X$ and $Y$ are $\alpha$-presentable, the weight $ N _h$ is $\alpha$-small.

\begin{es}
	When $\V=\bo{Cat}$ and $h\colon \mathbbm{2}_2\to \mathbbm{2}$ is the projection from the free category on a parallel pair to $\mathbbm{2}$, then limits weighted by $N_h$ correspond to equifiers (see \cite{lack20102}). 
\end{es}

Consider the case of $ N _X$. To give a diagram $H\colon \mathbbm{2}_{X+X}\to\K$ is the same as to give two objects $D_1,D_2$ and arrows $g_1,g_2\colon X\to \K(D_1,D_2)$. In that case, if $\K$ has enough limits, $\{ N _X,H\}$ can be seen as the equalizer:
\begin{center}
	\begin{tikzpicture}[baseline=(current  bounding  box.south), scale=2]
		
		\node (a) at (0,0) {$\{ N _{X},H\}$};
		\node (b) at (1.1,0) {$D_1$};
		\node (c) at (2.2,0) {$X\pitchfork D_2$};
		
		\path[font=\scriptsize]
		
		(a) edge [>->] node [above] {} (b)
		([yshift=1.5pt]b.east) edge [->] node [above] {$g_1^t$} ([yshift=1.5pt]c.west)
		([yshift=-1.5pt]b.east) edge [->] node [below] {$g_2^t$} ([yshift=-1.5pt]c.west);
		
	\end{tikzpicture}	
\end{center} 
where $g_1^t$ and $g_2^t$ are the transposes of $g_1$ and $g_2$. When $\K=\V$, to give an arrow $I\to \{ N _X,H\}$ is then equivalent to giving $x\colon I\to D_1$ for which the diagram
\begin{center}
	\begin{tikzpicture}[baseline=(current  bounding  box.south), scale=2]
		
		\node (a) at (-0.3,0) {$X$};
		\node (b) at (1,0.4) {$[D_1,D_2]\otimes D_1$};
		\node (c) at (1,-0.4) {$[D_1,D_2]\otimes D_1$};
		\node (d) at (2.3,0) {$D_2$};
		
		\path[font=\scriptsize]
		
		(a) edge [->] node [above] {$g_1\otimes x\ \ $} (b)
		(a) edge [->] node [below] {$g_2\otimes x\ \ $} (c)
		(c) edge [->] node [below] {ev} (d)
		(b) edge [->] node [above] {ev} (d);
		
	\end{tikzpicture}	
\end{center} 
commutes. We are now ready to prove the following:

\begin{prop}\label{flatimpliesfiltered}
	Let $ M \colon \C^{op}\to\V$ be a weight for which $ M *-\colon [\C,\V]\to \V$ preserves $\alpha$-small limits of representables. Then
	$$  M \cong\tx{colim} \left(\tx{El}( M _I)_{\V} \stackrel{\pi_\V}{\longrightarrow} \C \stackrel{Y}{\longrightarrow} [\C^{op},\V]\right) $$
	and so $ M $ is an $\alpha$-filtered colimit of representables.
\end{prop}
\begin{proof}
	The category $\tx{El}( M _I)$ is $\alpha$-filtered by Corollary~\ref{filtered-elements}; thus we only need to prove that the canonical map $\colim(Y\circ \pi_\V)\to  M $ is invertible. Since colimits are computed pointwise it's enough to show that the canonical map 
	$c\colon \colim \C(C,\pi_\V-)\to M (C)$ is invertible for any $C\in\C$. 
	
	Consider now a strong generator $\G$ of $\V$ made of $\alpha$-presentable objects; then the morphism $c$ above is invertible if and only if $\V_0(X,c)$ is invertible for any $X\in\G$. Since every object of $\G$ is $\alpha$-presentable and $\tx{El}( M _I)$ is $\alpha$-filtered, the functor $\V_0(X,-)$ preserves the colimit $\colim \C(C,\pi_\V-)$. Therefore it suffices to show that the canonical map 
	$$ c_X\colon \tx{colim} \left(\tx{El}( M _I) \stackrel{\pi}{\xrightarrow{\hspace*{0.4cm}}} \C_0^{op} \stackrel{\V_0(X,\C(C,-)_0)}{\xrightarrow{\hspace*{1.9cm}}} \bo{Set}\right)\longrightarrow \V_0(X, M (C)) $$
	is an isomorphism for any $X\in\G$. Below we are going to consider the elements of the colimit on the left as equivalence classes defined in the standard way.\\ 
	Consider $f\colon X\to M  C$; then, since $X$-powers are $\alpha$-small limits, we have 
	$$ \V_0(X, M (C))\cong \V_0(I,[X, M (C)])\cong\V_0(I, M *(X\pitchfork \C(C,-)))\ ;$$
	by Lemma~\ref{goodunit} we find $D\in\C$, $x\colon I\to M  D$, and $g\colon X\to\C(C,D)$ such that $f$ coincides with
	$$ X\stackrel{g\otimes x}{\xrightarrow{\hspace*{0.9cm}}}\C(C,D)\otimes M  D\stackrel{ M \otimes id }{\xrightarrow{\hspace*{0.9cm}}}[ M  D, M  C]\otimes  M  D\stackrel{ev}{\longrightarrow} M  C. $$ 
	In other words $c_X[g,x]=f$, so that $c_X$ is surjective.
	To prove the injectivity of $c_X$ we need the $\alpha$-small weight $ N _X$ introduced in Definition \ref{genequalizer}. Consider $g_i\colon X\to\C(C,D_i)$ and $x_i\colon I\to M  D_i$, for $i=1,2$, such that $c_X[g_1,x_1]=c_X[g_2,x_2]$; we need to prove that $[g_1,x_1]=[g_1,x_1]$. First, since $\tx{El}( M _I)$ is filtered we can assume that $D=D_1=D_2$ and $x=x_1=x_2$. Now $g_1,g_2\colon X\to\C(C,D)$ determine a diagram $H\colon \mathbbm{2}_{X+X}\to\C$, and $x$ corresponds to an arrow $\bar{x}\colon I\to \{ N _X, M  H\}$ (see comments above). Since $ M *-$ preserves $\alpha$-small limits of representables we obtain
	$$ \{ N _X, M  H^{op}\}\cong \{ N _X, M * YH\}\cong  M *\{ N _X,YH\}\ ;$$
	then, using Lemma~\ref{goodunit} again, we find that $\bar{x}$ factors through $y\colon I\to M  E$ and $h\in\C_0(D,E)$, for some $E\in\C$. This means that $ M (h)y=x$ and $ \C(C,h)\circ g_1= \C(C,h)\circ g_2$. Thus $[g_1,x_1]= [g_1,x_1]$ as desired and $c_X$ is an isomorphism.
\end{proof}

\begin{obs}
	When $\V_0(I,-)$ is cocontinuous and strong monoidal, as in most of the examples, the proof becomes simpler. Consider any $\alpha$-presentable $X$ in $\V_0$; then 
	\begin{equation*}
		\begin{split}
			\V_0(X, M (C))&\cong \V_0(I,[X, M *\C(C,-)])\\
			&\cong \V_0(I, M *(X\pitchfork\C(C,-)))\\
			&\cong  M _I*(X\pitchfork\C(C,-)_I)\\
			&\cong  M _I*\V_0(X,\C(C,-)_0)\\
			&\cong \tx{colim} \ \left( \tx{El}( M _I)\stackrel{\pi}{\xrightarrow{\hspace*{0.6cm}}} \C_0\stackrel{\C(C,-)_0}{\xrightarrow{\hspace*{1cm}}} \V_0\stackrel{\V_0(X,-)}{\xrightarrow{\hspace*{1cm}}}\bo{Set}\right)\\
			&\cong \V_0\left(X,\tx{colim} \left(\tx{El}( M _I)\stackrel{ M }{\xrightarrow{\hspace*{0.6cm}}} \C_0^{op}\stackrel{\C(C,-)_0}{\xrightarrow{\hspace*{1cm}}} \V_0\right)\right).\\
		\end{split}
	\end{equation*}
	Since the $\alpha$-presentable objects form a strongly generating family the result follows.
\end{obs}

Recall that we assume $\V$ to satisfy the conditions $(a)$ and $(b)$ from the beginning of this section; in this context we can characterize the $\alpha$-flat $\V$-functors as follows:

\begin{prop}\label{flat=filtered}
	Let $ M \colon \C^{op}\to\V$ be a weight; the following are equivalent:\begin{enumerate}\setlength\itemsep{0.25em}
		\item $ M $ is $\alpha$-flat;
		\item $ M *-\colon [\C,\V]\to \V$ preserves $\alpha$-small limits of representables;
		\item $ M $ is a (conical) $\alpha$-filtered colimit of representables.
	\end{enumerate}
\end{prop}
\begin{proof}
	$(1)\Rightarrow (2)$ is trivial, $(2)\Rightarrow (3)$ is a direct consequence of the previous Proposition, while $(3)\Rightarrow (1)$ follows from the fact that representable functors are $\alpha$-flat and these are closed under $\alpha$-filtered colimits in $[\C,\V]$.
\end{proof}

And as a consequence:

\begin{teo}\label{flat-preservation}
	A $\V$-category $\A$ has $\alpha$-flat colimits if and only if it has $\alpha$-filtered colimits. A $\V$-functor from such an $\A$ preserves $\alpha$-flat colimits if and only if it preserves $\alpha$-filtered colimits.
\end{teo}
\begin{proof}
	One direction is trivial. Conversely assume that $\A$ has all $\alpha$-filtered colimits and let $ M \colon \C^{op}\to\V$ be an $\alpha$-flat weight. By Proposition~\ref{flat=filtered} we can write $ M \cong\colim(YK)$ where $Y\colon \C\to[\C^{op},\V]$ is Yoneda and $K\colon \D_V\to\C$ in indexed on an $\alpha$-filtered category $\D$. Therefore, given any $H\colon \C\to\A$, we obtain a chain of isomorphisms (either side existing if the other does):
	\begin{equation*}
		\begin{split}
			 M *H&\cong (\tx{colim}YK)* H\\
			&\cong \tx{colim} (YK*H)\\
			&\cong \tx{colim} (HK).\\
		\end{split}
	\end{equation*}
	Thus the $ M *H$ exists since $\A$ has $\alpha$-filtered colimits. For the same reason a functor $F\colon \A\to\K$ preserves $ M *H$ if and only if it preserves $\tx{colim}(HK)$.
\end{proof}

\begin{teo}\label{2-cataccessibility}
	A $\V$-category $\A$ is $\alpha$-accessible if and only if it is conically $\alpha$-accessible.
\end{teo}
\begin{proof}
	By the theorem above an object $A$ of $\A$ is $\alpha$-presentable if and only if is conically $\alpha$-presentable, so that $\A_{\alpha}=\A_\alpha^c$. Arguing as above, for any $\alpha$-flat $ M \colon \C^{op}\to\V$ and a diagram $H\colon \C\to\A_{\alpha}\subseteq\A$, the colimit $ M *H$ can be replaced by an $\alpha$-filtered one $\tx{colim} (HK)$, which still lands in $\A_\alpha$. Thus an object is an $\alpha$-flat colimit of $\alpha$-presentables if and only if it is an $\alpha$-filtered colimit of (conically) $\alpha$-presentables. The result then follows.
\end{proof}

A notion of $\V$-sketch was introduced in \cite[Section~7]{BQR98}; there it was proven that being accessible is equivalent to being the $\V$-category of models of a $\V$-sketch. Putting this together with our result above we obtain: 

\begin{teo}
	Let $\A$ be a $\V$-category; the following are equivalent:\begin{enumerate}
		\item $\A$ is accessible;
		\item $\A$ is conically accessible;
		\item $\A$ is equivalent to the $\V$-category of models of a $\V$-sketch.
	\end{enumerate}
\end{teo}

\begin{cor}
	A $\V$-category is Cauchy complete if and only if idempotents split.
\end{cor}
\begin{proof}
	If $\C$ is Cauchy complete as a $\V$-category then it certainly has splittings of idempotents. Conversely, if $\C$ has splittings of idempotents then it has all those conical colimits indexed on ordinary absolute diagrams. Let $ M \colon\B^{op}\to\V$ be a Cauchy weight; this means that $ M $ is $\alpha$-flat for each $\alpha$. By Proposition~\ref{flatimpliesfiltered} the ordinary category $\tx{El}( M _I)$ is then $\alpha$-filtered for each $\alpha$, and hence absolute in the ordinary sense. Arguing as above, $ M $-weighted colimits in $\C$ can be reduced to conical colimits indexed on $\tx{El}( M _I)$. It follows then that $\C$ has $ M $-weighted colimits and therefore is Cauchy complete.
\end{proof}

\begin{obs}
	When $\V_0(I,-)$ is strong monoidal, the result above is given by \cite[Proposition~3.2]{Nik17:articolo}.
\end{obs}

\subsection{The cartesian closed case}

We now give a more explicit characterization of $\alpha$-flat $\V$-functors in the case where $\V_0$ is endowed with the cartesian monoidal structure and satisfies condition $(a)$ from before (condition $(b)$ is automatic).

In this case, for each $X\in\V_0$ the functor $\V_0(X,-)$ is (strong) monoidal; therefore induces a change of base 2-functor $ M_X\colon \V\tx{-}\bo{Cat}\longrightarrow\bo{Cat}$.
Hence for every weight $ M \colon \C^{op}\to\V$ we obtain an ordinary functor 
$$ M _X\colon M_X\C^{op}\stackrel{M_X M }{\xrightarrow{\hspace*{1cm}}}\M_X\V\stackrel{M_X\C(I,-)}{\xrightarrow{\hspace*{1.3cm}}}\bo{Set}$$
generalizing an earlier notation $ M _I$ for $\V_0(I, M _0-)\colon \C_0^{op}\to\bo{Set}$.

\begin{obs}\label{X-elements-0}
	For each $X\in\V$ and $ M \colon \C^{op}\to\V$ we consider the category of elements $\tx{El}( M _X)$; this can be described explicitly as follows:\begin{itemize}\setlength\itemsep{0.25em}
		\item an object is a pair $(C\in\C,x\colon X\to M  C)$;
		\item a morphism $f\colon (C,x)\to (D,y)$ is an arrow $f\colon X\to\C(C,D)$ for which the triangle
		\begin{center}
			\begin{tikzpicture}[baseline=(current  bounding  box.south), scale=2]

				\node (b0) at (1.5,0.8) {$\C(C,D)\times M (D)$};
				\node (c0) at (0,0.8) {$X$};
				\node (d0) at (1.5,0) {$ M (C)$};
				
				\path[font=\scriptsize]
				
				(c0) edge [->] node [above] {$(f,y)$} (b0)
				(b0) edge [->] node [right] {$\tx{ev}_ M $} (d0)
				(c0) edge [->] node [below] {$x\ \ $} (d0);
				
			\end{tikzpicture}	
		\end{center}
		commutes (remember that $ M $ is contravariant);
		\item for each $(C,x)$ the identity is given by $\tx{id}_{(C,x)}\colon X\stackrel{!}{\to} 1 \stackrel{1_C}{\longrightarrow}\C(C,C)$;
		\item composition is as follows: given $f\colon (C,x)\to (D,y)$ and $g\colon (D,y)\to (E,z)$ the composite $g\circ f$ is 
		$$  X \stackrel{(g,f)}{\xrightarrow{\hspace*{0.9cm}}} \C(D,E)\times\C(C,D)\stackrel{M}{\xrightarrow{\hspace*{0.4cm}}}\C(C,E) $$ 
		where $M$ is the composition map in $\C$.
	\end{itemize}
\end{obs}

Note that $\tx{El}( M _X)$ is not in general $\tx{El}(\V_0(X, M _0-))$: they have same objects but a morphism in $\tx{El}(\V_0(X, M _0-))$ from $(C,x)$ to $(D,y)$ is a morphism $1\to\C(C,D)$ such that the induced $X\to1\to\C(C,D)$ defines a morphism in $\tx{El}( M _X)$. Finally observe that for each $X$ there is an induced functor 
$$ J_X\colon \tx{El}( M _1)\longrightarrow \tx{El}( M _X)$$
which acts by precomposition with the unique morphism $!\colon X\to 1$.

\begin{prop}\label{cartesian-flat-characterization}
	Let $\V=(\V_0,\times,1)$ be as above and $\G\subseteq\V_\alpha$ a strong generator. A $\V$-functor $ M \colon \C^{op}\to\V$ is $\alpha$-flat if and only if:\begin{enumerate}\setlength\itemsep{0.25em}
		\item $\tx{El}( M _1)$ is $\alpha$-filtered;
		\item for each $X\in\G$ the functor $J_X\colon \tx{El}( M _1)\longrightarrow \tx{El}( M _X)$ is final.
	\end{enumerate}
\end{prop}
\begin{proof}
	By Propositions \ref{flatimpliesfiltered} and \ref{flat=filtered} above, $ M $ is $\alpha$-flat if and only if $\tx{El}( M _1)$ is $\alpha$-filtered and for any $C\in\C$ the following isomorphism holds
	$$  M (C)\cong\tx{colim} \left(\tx{El}( M _1) \stackrel{\pi}{\xrightarrow{\hspace*{0.5cm}}} \C_0 \stackrel{\C(C,-)_0}{\xrightarrow{\hspace*{1cm}}} \V_0\right). $$
	Equivalently, since $\G$ is a strong generator, $ M $ is $\alpha$-flat if and only if $\tx{El}( M _1)$ is $\alpha$-filtered and for any $C\in\C$ and $X\in\G$ we have
	$$ \V_0(X, M (C))\cong\tx{colim} \left(\tx{El}( M _1) \stackrel{\pi}{\xrightarrow{\hspace*{0.5cm}}} \C_0 \stackrel{\V_0(X,\C(C,-)_0)}{\xrightarrow{\hspace*{1.9cm}}} \bo{Set}\right). $$
	Bearing in mind that $\tx{El}( M _1)$ is $\alpha$-filtered, the isomorphism above holds if and only if for all $X\in\G$ and $C\in\C$: \begin{itemize}\setlength\itemsep{0.25em}
		\item for any $x\colon X\to M  C$ there exist $D\in\C$, $y\colon 1\to M  D$, and $g\colon X\to\C(C,D)$ such that $x$ coincides with the composite
		$$ X\stackrel{g\times x}{\xrightarrow{\hspace*{0.9cm}}}\C(C,D)\times M  D\stackrel{\tx{ev}_ M }{\xrightarrow{\hspace*{0.9cm}}} M  C. $$
		In other words,  such that $g$ defines a morphism $(C,x)\to J_X(D,y)$ in $\tx{El}( M _X)$. 
		\item for any $(y_1,D_1,g_1)$ and $(y_2,D_2,g_2)$ as above there exist $E\in\C$, $z\colon 1\to M  E$, and maps $h_i\colon D_i\to E$ such that $ M _I(h_i)(z)=x_i$ and $ \C(h_1,C)\circ g_1= \C(h_2,C)\circ g_2$. In other words, such that we have a commutative square
		\begin{center}
			\begin{tikzpicture}[baseline=(current  bounding  box.south), scale=2]
				
				\node (a) at (-0.3,0) {$(C,x)$};
				\node (b) at (1,0.4) {$J_X(D_1,y_1)$};
				\node (c) at (1,-0.4) {$J_X(D_2,y_2)$};
				\node (d) at (2.3,0) {$J_X(E,z)$};
				
				\path[font=\scriptsize]
				
				(a) edge [->] node [above] {$g_1$} (b)
				(a) edge [->] node [below] {$g_2$} (c)
				(c) edge [->] node [below] {$h_1$} (d)
				(b) edge [->] node [above] {$h_2$} (d);
				
			\end{tikzpicture}	
		\end{center} 
		in $\tx{El}( M _X)$.
	\end{itemize} 
	These conditions are easily seen to be equivalent to $(2)$.
\end{proof}

As a consequence all the categories $\tx{El}( M _X)$, for $X\in\G$, are also $\alpha$-filtered (see Remark~\ref{final-filtered-subcat}).

\subsection{2-categories}

We now further specialize to the case $\V=\bo{Cat}$ and give a characterization of $\alpha$-flat 2-functors in terms of a certain double category.

\begin{Def}[\cite{GP1999:articolo}]
	Let $ M \colon \C^{op}\to\bo{Cat}$ be a 2-functor. The double category of elements $\mathbb{E}\tx{l}( M )$ of $ M $ is the one with:\begin{itemize}\setlength\itemsep{0.25em}
		\item objects: pairs $(C,x)$ with $C\in\C$ and $x\in M (C)$;
		\item horizontal arrows $f\colon (C,x)\to(D,y)$: morphisms $f\colon C\to D$ in $\C$ with $ M (f)(y)=x$;
		\item vertical arrows $\xi\colon (C,x)\todot(C,x')$: morphisms $\xi\colon x\to x'$ in $ M (C)$.
		\item double cells:
		\begin{center}
			\begin{tikzpicture}[baseline=(current  bounding  box.south), scale=2]
				
				\node (a) at (0,0) {$(C,x)$};
				\node (b) at (1.1,0) {$(D,y)$};
				\node (a'') at (0,-0.45) {$\bullet$};
				\node (b'') at (1.1,-0.45) {$\bullet$};
				\node (c') at (0.6,-0.45) {$\Downarrow \alpha$};
				\node (a') at (0,-0.9) {$(C,x')$};
				\node (b') at (1.1,-0.9) {$(D,y')$};

				\path[font=\scriptsize]
				
				(a) edge[->] node [above] {$f$} (b)
				(a') edge[->] node [below] {$g$} (b')
				(a) edge[->] node [left] {$\xi$} (a')
				(b) edge[->] node [right] {$\mu$} (b');
				
			\end{tikzpicture}
		\end{center}
		2-cells $\alpha\colon f\Rightarrow g$ in $\C$ for which $ M (\alpha)(\mu)=\xi$.
	\end{itemize}
\end{Def}

{\bf Notation:} Let $\mathbb{D}$ be a double category.\begin{itemize}
	\item We denote by $H(\mathbb{D})$ the {\em horizontal} category of $\mathbb{D}$; this has the same objects as $\mathbb{D}$ and the horizontal arrows as morphisms. 
	\item We denote by $H_1(\mathbb{D})$ the category with vertical arrows as objects and double cells as morphisms, endowed with the horizontal composition of cells. 
\end{itemize} 

It follows that the horizontal category $H(\mathbb{E}\tx{l}( M ))$ of $\mathbb{E}\tx{l}( M )$ corresponds to $\tx{El}( M _1)$, where $ M _1:=\bo{Cat}_0(1, M _0-)\colon \C_0\to\bo{Set}$ as usual.
For any $\mathbb{D}$ there is a functor $1_H\colon \colon H(\mathbb{D})\to H_1(\mathbb{D})$ which sends an object $D$ to the vertical identity $D\todot D$, and an arrow $f\colon D\to C$ to the identity 2-cell $1_f$. 

\begin{Def}\label{final-horizontal}
	Let $\mathbb{D}$ be a double category; we say that the horizontal category $H(\mathbb{D})$ of $\mathbb{D}$ is {\em final} in $\mathbb{D}$ if the ordinary functor $1_H\colon H(\mathbb{D})\to H_1(\mathbb{D})$ is final. 
\end{Def}

\begin{obs}
	If $H(\mathbb{D})$ is final in $\mathbb{D}$, then double colimits indexed on $\mathbb{D}$ are the same as ordinary colimits on $H(\mathbb{D})$ in the following sense:
	
	Let $\C$ be a 2-category and $\mathbb{K}\colon \mathbb{D}\to\C$ a double functor (where $\C$ is seen as a double category with identity vertical morphisms). Let $K\colon H(\mathbb{D})\to\C$ be the induced horizontal functor; then the double colimit of $\mathbb{K}$ exists in $\C$ if and only if the conical colimit of $K$ does so, and in that case they coincide.
\end{obs}

\begin{Def}
	We say that a double category $\mathbb{D}$ is {\em $\alpha$-filtered} if the horizontal category $H(\mathbb{D})$ is $\alpha$-filtered and final in $\mathbb{D}$. Equivalently, $\mathbb{D}$ is $\alpha$-filtered if and only if:\begin{enumerate}
		\setlength\itemsep{0.25em}
		\item $H(\mathbb{D})$ is $\alpha$-filtered;
		\item for any vertical morphism $\xi\colon C\todot D$ there exists a square as below;
		\begin{center}
			\begin{tikzpicture}[baseline=(current  bounding  box.south), scale=2]
				
				\node (a) at (0,0) {$C$};
				\node (b) at (1.1,0) {$B$};
				\node (a'') at (0,-0.45) {$\bullet$};
				\node (b'') at (1.1,-0.45) {$\bullet$};
				\node (c') at (0.6,-0.45) {$\Downarrow \alpha$};
				\node (a') at (0,-0.9) {$D$};
				\node (b') at (1.1,-0.9) {$B$};

				\path[font=\scriptsize]
				
				(a) edge[->] node [above] {$f$} (b)
				(a') edge[->] node [below] {$g$} (b')
				(a) edge[->] node [left] {$\xi$} (a')
				(b) edge[->] node [right] {$id$} (b');
				
			\end{tikzpicture}
		\end{center}
		\item for any pair of double cells $\alpha$ and $\beta$
		\begin{center}
			\begin{tikzpicture}[baseline=(current  bounding  box.south), scale=2]
				
				\node (a) at (0,0) {$C$};
				\node (b) at (1.1,0) {$B$};
				\node (a'') at (0,-0.45) {$\bullet$};
				\node (b'') at (1.1,-0.45) {$\bullet$};
				\node (c') at (0.62,-0.45) {$\Downarrow \alpha,\beta$};
				\node (a') at (0,-0.9) {$D$};
				\node (b') at (1.1,-0.9) {$B$};

				\path[font=\scriptsize]
				
				(a) edge[->] node [above] {$f$} (b)
				(a') edge[->] node [below] {$g$} (b')
				(a) edge[->] node [left] {$\xi$} (a')
				(b) edge[->] node [right] {$id$} (b');
				
			\end{tikzpicture}
		\end{center}
		there exists a horizontal arrow $h\colon B\to A$ such that $\alpha h=\beta h$.
	\end{enumerate}
\end{Def}

\begin{prop}
	Let $\C$ be a 2-category and $ M \colon \C\to\bo{Cat}$ a 2-functor; the following are equivalent:\begin{enumerate}\setlength\itemsep{0.25em}
		\item $ M $ is $\alpha$-flat;
		\item the double category $\mathbb{E}\tx{l}( M )$ is $\alpha$-filtered.
	\end{enumerate}
\end{prop}
\begin{proof}
	Note that, by Corollary~\ref{cartesian-flat-characterization} applied to the strong generator $\{\mathbbm{2}\}$, the 2-functor $ M $ is $\alpha$-flat if and only if $\tx{El}( M _1)$ is $\alpha$-filtered and the functor  $J\colon \tx{El}( M _1)\longrightarrow \tx{El}( M _\mathbbm{2})$, induced by precomposition with $!\colon \mathbbm{2}\to 1$, is final. But $\tx{El}( M _1)=H(\mathbb{E}\tx{l}( M ))$ is the underlying category spanned by the horizontal arrows of $\mathbb{E}\tx{l}( M )$, and $\tx{El}( M _\mathbbm{2})=H_1(\mathbb{E}\tx{l}( M ))$ is the ordinary category with vertical arrows as objects and double cells as morphisms. Under this interpretation the functor $J$ is the same as $1_H\colon H(\mathbb{E}\tx{l}( M ))\to H_1(\mathbb{E}\tx{l}( M ))$; therefore the result follows by the definition of $\alpha$-filtered double category.
\end{proof}

\section{When flat equals filtered plus absolute}\label{flat=absolute+filtered}

In Section~\ref{dualizable} we briefly recall the notion of dualizable object, then in Section~\ref{locallydualiza} we introduce what we are calling the {\em locally dualizable categories}. Finally in Section~\ref{theorem-locallydualiz} we show that if $\V$ is locally dualizable then $\alpha$-flat colimits are generated by absolute colimits and the usual $\alpha$-filtered ones. Key examples of locally dualizable categories can be found in Example~\ref{locduaexample} below. 

\subsection{Dualizable objects}\label{dualizable}

\begin{Def}
	We say that an object $P$ of $\V$ is {\em dualizable} if there exist $P^*\in\V$ and morphisms $\eta_P\colon I\to P\otimes P^*$ and $\epsilon_P\colon P^*\otimes P\to I$, called unit and counit respectively, satisfying the triangle equalities.
	In that case $P^*$ is unique up to isomorphism and is called the {\em dual} of $P$. 
\end{Def}

Note that the unit $I$ is always dualizable and that if $P$ is dualizable then $P^*$ is too with $(P^*)^*\cong P$. When, as we assume here, $\V$ is symmetric monoidal closed, $P$ is dualizable if and only if there exists $P^*\in\V$ such that $ [P,-] \cong P^*\otimes -\colon \V_0\to\V_0$.

The following is a well-known result about powers and copowers by dualizable objects:

\begin{prop}
	Powers and copowers by dualizable objects are absolute; moreover for any $A\in\A$ and any dualizable $P$ we have $P\pitchfork \A\cong P^*\cdot A$ (either side existing if the other does).
\end{prop}
\begin{proof}
	Let $P$ be dualizable in $\V$; then $P\otimes -\cong [P^*,-]$ is continuous and hence copowers by $P$ are dualizable by \cite[Theorem~6.22]{KS05:articolo}. About the last statement, consider the following isomorphisms natural in $B\in\A$
	\begin{equation*}
		\begin{split}
			\A(B,P\pitchfork A)&\cong [P,\A(B,A)]\\
			&\cong P^*\otimes \A(B,A)\\
			&\cong \A(B,P^*\cdot A)\\
		\end{split}
	\end{equation*}
	where the last isomorphism holds since copowers by $P^*$ are absolute; therefore $P\pitchfork A\cong P^*\cdot A$. Finally, powers by $P$ are absolute because they are the same as copowers by $P^*$.
\end{proof}

\subsection{Locally dualizable categories}\label{locallydualiza}

We can now introduce the bases of enrichment considered in this section:

\begin{Def}\label{locdualiz}
	Let $\V=(\V_0,\otimes,I)$ be a cocomplete symmetric monoidal closed category; we say that $\V$ is {\em locally dualizable} if: 
	\begin{enumerate}\setlength\itemsep{0.25em}
		\item[(a)] $\V_0$ has finite direct sums;
		\item[(b)] The unit $I$ is regular projective and finitely presentable;
		\item[(c)] $\V_0$ has a strong generator $\G$ made of dualizable objects;
		\item[(d)] for every $A,B\in\V$ and every arrow $z\colon I\to A\otimes B$ there exists a dualizable object $P\in\V$ and maps $x\colon P\to A$ and $y\colon P^*\to B$ such that
		\begin{center}
			\begin{tikzpicture}[baseline=(current  bounding  box.south), scale=2]

				\node (b0) at (1.3,0.7) {$P\otimes P^*$};
				\node (c0) at (0,0.7) {$I$};
				\node (d0) at (1.3,0) {$A\otimes B$};
				
				\path[font=\scriptsize]
				
				(c0) edge [->] node [above] {$\eta_P$} (b0)
				(b0) edge [->] node [right] {$x\otimes y$} (d0)
				(c0) edge [->] node [below] {$z\ $} (d0);
				
			\end{tikzpicture}	
		\end{center}
		commutes. 
	\end{enumerate}
\end{Def}

{\bf Notation:} From now on we write simply $x\stackrel{\scriptscriptstyle{P}}{\otimes}y$ in place of the composite $(x\otimes y)\circ\eta_P$, for any dualizable object $P$, $x\colon P\to A$, and $y\colon P^*\to B$. Note in particular that if $P=I$ then $x\stackrel{\scriptscriptstyle{I}}{\otimes} y$ is just $x\otimes y$ up to the isomorphism $I\otimes I\cong I$.

By condition $(a)$, we know in particular that $\V_0$ is the underlying category of a $\bo{CMon}$-enriched category $\bar{\V}$; this notation will also be used later on.

For every $P\in \G$ we have $[P,-]\cong P^*\otimes-$, so that $[P,-]$ is cocontinuous and hence $\V_0(P,-)\cong \V_0(I,[P,-])$ is finitary and preserves regular epimorphisms. Therefore $\G$ is a strong generator made of finitely presentable regular projective objects and $\V_0$ is hence a finitary quasivariety \cite[Definition~4.4]{LT20:articolo}, and in particular locally finitely presentable. Moreover, for any $P,Q\in\G$, the functor $\V_0(P\otimes Q,-)\cong \V_0(P,[Q,-])$ is finitary and preserves regular epimorphisms as well. This means that $\V$ is actually a symmetric monoidal finitary quasivariety in the sense of \cite[Definition~4.11]{LT20:articolo}, and in particular $\V$ is locally finitely presentable as a closed category.

\begin{prop}\label{four-prime}
	Assume that $\V=(\V_0,I,\otimes)$ satisfies the conditions $(a),(b)$ and $(c)$ above. Then $(d)$ holds if and only if:
	\begin{itemize}
		\item[$(d^*)$] for every $A,B\in\V$, any arrow $z\colon I\to A\otimes B$ can be written as 
		$$ z=\sum_{i=1}^n (x_i\stackrel{\scriptscriptstyle{P_i}}{\otimes} y_i)$$
		for some $P_i\in\G$, $x_i\colon P_i\to A$, and $y_i\colon P_i^*\to B$. 
	\end{itemize}
	When writing the sum above we are seeing $\V_0$ as the $\bo{CMon}$-category $\bar{\V}$.
\end{prop}
\begin{proof}
	Fix $A,B\in\V$ and an arrow $z\colon I\to A\otimes B$. If $(d^*)$ holds then it's enough to consider $P:=\textstyle\bigoplus_{i=1}^nP_i$, $x:=\textstyle\sum_{i=1}^nx_i$, and $y:=\textstyle\sum_{i=1}^ny_i$. Then $P$ is still dualizable, with dual $P^*\cong \textstyle\bigoplus_{i=1}^nP^*_i$, and by construction $x\stackrel{\scriptscriptstyle{P}}{\otimes} y=z$.
	
	Conversely, assume that $(d)$ holds. Since $\G$ is a strong generator made of finitely presentable and projective objects, and $P$ is finitely presentable and regular projective as well, we can find $P_1,\dots,P_n\in\G$ and a split epimorphism $q\colon Q:=\textstyle\bigoplus_{i=1}^nP_i\twoheadrightarrow P$, with section $s\colon P\rightarrowtail Q$ (see for example \cite[Proposition~4.8.(4)]{LT20:articolo}). It follows that the following triangle commutes
	\begin{center}
		\begin{tikzpicture}[baseline=(current  bounding  box.south), scale=2]

			\node (b0) at (1.3,0.8) {$\ Q\otimes Q^*$};
			\node (c0) at (0,0.8) {$I$};
			\node (d0) at (1.3,0) {$\ P\otimes P^*$};
			
			\path[font=\scriptsize]
			
			(c0) edge [->] node [above] {$\eta_{Q} $} (b0)
			(b0) edge [->] node [right] {$q\otimes s^*$} (d0)
			(c0) edge [->] node [below] {$\eta_P\ $} (d0);
			
		\end{tikzpicture}	
	\end{center}
	in other words $\eta_P=q\stackrel{\scriptscriptstyle{Q}}{\otimes}  s^*$. Therefore $z=x\stackrel{\scriptscriptstyle{P}}{\otimes}  y= (x\circ q)\stackrel{\scriptscriptstyle{Q}}{\otimes}  (y\circ s^*)$. Now it's enough to define $x_i$ and $y_i$ as the $i$-th components of $(x\circ q)$ and $(y\circ s^*)$ respectively.
\end{proof}

\begin{obs}
	Assume that $\V$ satisfies conditions $(a),(b),$ and $(c)$. Define the map
	$$q_{A,B}:=\sum (x\stackrel{\scriptscriptstyle{P}}{\otimes}  y)\colon \sum_{(P,x\colon P\to A,\\ y\colon P^*\to B)}I\longrightarrow A\otimes B. $$
	It's then easy to see that $\V$ satisfies $(d^*)$ if and only if $\V_0(I,q_{A,B})$ is a surjection. It is not true in general, not even for graded abelian groups, that the map $q_{A,B}$ is a regular epimorphism in $\V_0$. What is true in some cases, which include graded abelian groups, is that the following
	$$\sum_{(P,Q,x\colon P\to A,\\ y\colon P^*\otimes Q\to B)}Q\longrightarrow A\otimes B $$
	is a regular epimorphism in $\V_0$; where the component $(P,Q,x,y)$ is given by the composite
	$$ Q\cong I\otimes Q\stackrel{i_P}{\longrightarrow}(P\otimes P^*)\otimes Q\cong P\otimes (P^*\otimes Q)\stackrel{x\otimes y}{\xrightarrow{\hspace*{0.8cm}}} A\otimes B. $$
	But condition $(d)$ is all we need.
\end{obs}

\begin{prop}
	Let $\V$ be locally dualizable with strong generator $\G$, and let $\C$ be a small compact closed $\V$-category. Then $\W_0:=[\C,\V]_0$, endowed with Day's convolution as tensor products, is locally dualizable with strong generator $$\G':=\{P\cdot Yg\ |\ P\in\G, g\in\C\},$$ where $Yg=\C(g,-)$. 
\end{prop}
\begin{proof}
	The category $\W$ is symmetric monoidal closed and cocomplete by construction.
	Observe that the unit of $\W_0$ is $Ye=\C(e,-)$, where $e$ is the unit of $\C$; this is regular projective and finitely presentable because
	$$ \W_0(Ye,-)\cong \V_0(I,[\C,\V](Ye,-))\cong \V_0(I,\tx{ev}_e-) $$
	and the unit $I$ of $\V$ is such. Moreover $\W$ has finite direct sums since $\V$ has them and the representables have duals $(Yg)^*=Y(g^*)$. For each $P\in\G$ and $A\in\W$ we have
	$$ [\C,\V]_0(P\cdot Yg,A)\cong \V_0(P,[\C,\V](Yg,A))\cong \V_0(P,A(g)).$$
	In particular an arrow $z\colon Ye\to A\otimes B$ in $\W$ corresponds to an arrow $\bar{z}\colon I\to (A\otimes B)(e)$ in $\V$. Now note that by definition 
	\begin{equation*}
		\begin{split}
			(A\otimes B)(e) &\cong \int^{g,h}\C(g\otimes h,e)\otimes A(g)\otimes B(h)\\
			&\cong \int^{g,h}\C(h,g^*)\otimes A(g)\otimes B(h)\\
			&\cong \int^{g}A(g)\otimes B(g^*)\\
		\end{split}
	\end{equation*}
	and that moreover we have a regular epimorphism in $\V$
	$$ \sum_{g}A(g)\otimes B(g^*) \twoheadrightarrow \int^{g}A(g)\otimes B(g^*). $$
	Since $I$ is finitely presentable and projective, $\bar{z}$ factors as a finite direct sum $\bar{z}=\textstyle\sum_i\bar{y}_{i}$ for some $\bar{y}_{i}\colon I\to A(g_i)\otimes B(g_i^*)$; by hypothesis we have $\bar{y}_i= \bar{x}_{i}\stackrel{\scriptscriptstyle{P_i}}{\otimes}  \bar{y_{i}}$ for some dualizable $P_{i}$ in $\V$, $\bar{x}_{i}\colon P_{i}\to A(g_i)$ and $\bar{y}_{i}\colon P_{i}^*\to B(g_i^*)$. Let $x_{i}\colon P_{i}\cdot Yg_i\to A$ and $y_{i}\colon (P_{i}\cdot Yg_i)^*\cong (P_{i}^*\cdot Y(g_i^*))\to B$ be the induced morphisms in $\W$. Then $$z=\sum_{i} x_{i}\stackrel{\scriptscriptstyle{P_{i}\cdot Yg_i}}{\otimes}  y_{i}, $$ and therefore $\W$ is locally dualizable by Proposition~\ref{four-prime} with strong generator $\G':=\{P\cdot Yg\ |\ P\in\G, g\in\C\}$.
\end{proof}

As a consequence, for any compact closed $\bo{CMon}$-category $\C$ the presheaf category $[\C^{op},\bo{CMon}]$ is locally dualizable with strong generator given by the representables.

All the examples below can be constructed as above starting from the first.

\begin{ese}\label{locduaexample}$ $\
	\begin{itemize}\setlength\itemsep{0.25em}
		\item The symmetric monoidal category $\bo{CMon}$ of commutative monoids is locally dualizable with strong generator $\G=\{\mathbb{N}\}$.
		\item The symmetric monoidal category $\bo{Ab}$ of abelian groups; more generally the symmetric monoidal closed category $R\tx{-}\bo{Mod}$ of modules over a commutative ring $R$ with $\G=\{R\}$.
		\item The symmetric monoidal category $\bo{G}\tx{-}\bo{Gr}(R\tx{-}\bo{Mod})$ of $\bo{G}$-graded $R$-modules, for an abelian group $\bo{G}$ and a commutative ring $R$, with $\G=\{S_gR\}_{g\in\bo{G}}$.
	\end{itemize}
\end{ese}

\begin{prop}
	Let $\bar{\V}$ be a cocomplete symmetric monoidal closed $\bo{CMon}$-category; then $\V$ is locally dualizable if and only if there exists a small compact closed $\G$ and a monoidal adjunction
	\begin{center}
		
		\begin{tikzpicture}[baseline=(current  bounding  box.south), scale=2]

			\node (f) at (0,0.4) {$\bar{\V}$};
			\node (g) at (1.8,0.4) {$[\G^{op},\bo{CMon}]$};
			\node (h) at (0.66, 0.48) {$\perp$};
			\path[font=\scriptsize]

			([yshift=-1.3pt]f.east) edge [->] node [below] {$U$} ([yshift=-1.3pt]g.west)
			([yshift=2pt]f.east) edge [bend left,<-] node [above] {$F$} ([yshift=2pt]g.west);
		\end{tikzpicture}
		
	\end{center}
	(where $[\G^{op},\bo{CMon}]$ has Day's convolution as monoidal structure) with $U$ a conservative, filtered-colimit-preserving, and regular $\bo{CMon}$-functor for which the comparison maps
	$$ UA\otimes UB\to U(A\otimes B) $$
	are regular epimorphisms for any $A,B\in\V_0$. The last requirement is also equivalent to 
	$$ UFX\otimes UFY\to UF(X\otimes Y) $$
	being regular epimorphisms for any $X,Y\in[\G^{op},\bo{CMon}]$.
\end{prop}
\begin{proof}
	Assume first that $\V$ is locally dualizable; as specified at the beginning we can assume $\G$ to contain the unit and be closed in $\V_0$ under tensor product and duals. For any $P\in\G$ the $\bo{CMon}$-functor $\bar{\V}(G,-)$ preserves finite direct sums, filtered colimits, regular epimorphisms, and so also all coproducts. We can then consider the induced $\bo{CMon}$-functor $U=\bar{\V}(\G,1)\colon \bar{\V}\longrightarrow[\G^{op},\bo{CMon}]$ which is then conservative, continuous, and preserves filtered colimits, direct sums and regular epimorphism.  Moreover, since $\V_0$ is cocomplete, $U$ has a left adjoint $F$ which sends the representables to $\G$. This is a monoidal adjunction because $F$ is strong monoidal: the restriction of $F$ to $Y\G$ is strong monoidal by construction (being isomorphic to the identity) and hence, since $Y\G$ generates $[\G^{op},\bo{CMon}]$ under colimits, the tensor product is cocontinuous in each variable, and $F$ is itself cocontinuous, it follows that $F$ is strong monoidal too.
	We are only left to show that the comparisons $UA\otimes UB\to U(A\otimes B)$, induced by the monoidal structure on $U$, are regular epimorphisms in $[\G^{op},\bo{CMon}]$. This will be the case if and only if, for each $G\in\G$, the map $ (UA\otimes UB)G\to U(A\otimes B)G $ is a surjection (of monoids); that is, if and only if the canonical map
	$$ \int^{P,H\in\G}\G(G,P\otimes H)\otimes \bar{\V}(P,A)\otimes \bar{\V}(H,B)\longrightarrow\bar{\V}(G,A\otimes B)  $$
	is surjective. But $\G(G,P\otimes H)\cong\G(P^*\otimes G,H)$, so the coend on the left hand side can be rewritten as 
	$$ \int^{P\in\G}\bar{\V}(P,A)\otimes \bar{\V}(P^*\otimes G,B),  $$
	and this is covered by the coproduct of its components. As a consequence the morphism $ (UA\otimes UB)G\to U(A\otimes B)G $ is a regular epimorphism if and only if the induced map
	$$ \sum_{P\in\G} \bar{\V}(P,A)\otimes \bar{\V}(P^*\otimes G,B)\longrightarrow \bar{\V}(G,A\otimes B)$$
	is surjective. For that, given any $G\in\G$ and a map $f\colon G\to A\otimes B$, we can transpose $f$ to obtain $f^t\colon I\to A\otimes (B\otimes G^*)$. By hypothesis this can be expressed as $ f^t=\textstyle\sum_{i=1}^n (x_i\stackrel{\scriptscriptstyle{P_i}}{\otimes} y^t_i) $ for some $P_i\in\G$, $x_i\colon P_i\to A$ and $y^t_i\colon P_i^*\to B\otimes G^*$. Transposing again we can then write 
	\begin{equation}\label{G-sums}
		f=\sum_{i=1}^n (x_i\otimes y_i)\circ(\eta_{P_i}\otimes G)
	\end{equation}
	where $y_i\colon P_i^*\otimes G\to B$ is the transpose of $y_i^t$, and $i_{P_i}\otimes G\colon G\to (P_i\otimes P_i^*)\otimes G\cong P_i\otimes (P_i^*\otimes G)$. This proves that the desired map is a surjection of monoids.
	
	Assume conversely that we have such an adjunction. $\V_0$ has direct sums by hypothesis. Denote by $\G(-,J)$ the unit of $[\G^{op},\bo{CMon}]$; then $I\cong F\G(-,J)$ and $\V_0(I,-)\cong  \bo{CMon}(\mathbb{N},\tx{ev}_JU-)$ preserves filtered colimits and regular epimorphisms because $U$ does. Moreover, since $F$ is strong monoidal the image $F\G$ of the representables under $F$ consists of dualizable objects; these form a strong generator of $\V_0$ since $U$ is conservative and $\G$ is a strong generator of $[\G^{op},\bo{CMon}]$. It remains to check property $(d^*)$. 
	Consider $z\colon I\to A\otimes B$, since $I\cong F\G(-,J)$ this corresponds to a map $z^t\colon \G(-,J)\to U(A\otimes B)$. By hypothesis the comparison $UA\otimes UB\to U(A\otimes B)$ is a regular epimorphism; thus $z^t$ factors through it (since representables are projective), giving a map $z'\colon  \G(-,J)\to UA\otimes UB$. Since $[\G^{op},\bo{CMon}]$ is locally dualizable with strong generator given by the representables it follows that $ z'=\textstyle\sum_{i=1}^n (x^t_i\stackrel{\scriptscriptstyle{P_i}}{\otimes} y^t_i) $ for some $P_i\in\G$, $x^t_i\colon \G(-,P_i)\to UA$ and $y^t_i\colon \G(-,P_i)^*\to UB$. Transposing again through the adjunction we obtain maps $x_i\colon F\G(-,P_i)\to A$ and $y_i\colon F\G(-,P_i)^*\to B$ for which $ z=\textstyle\sum_{i=1}^n (x_i\stackrel{\scriptscriptstyle{P_i}}{\otimes}  y_i) $. Therefore $\V_0$ is locally dualizable.
	
	The last part of the statement is an easy consequence of the fact that whenever $U$ is conservative the counit of the adjunction is pointwise a strong epimorphism and in $\bo{CMon}$ strong and regular epimorphisms coincide.
	
\end{proof}

\begin{obs}
	When $\V_0$ is exact this adjunction is monadic by Duskin's monadicity theorem \cite[Theorem~9.1.3]{BW2000toposes}.
\end{obs}

An immediate consequence is:

\begin{prop}
	Let $\V=(\V_0,\otimes,I)$ be locally dualizable and $\bar{\W}$ be a symmetric monoidal closed $\bo{CMon}$-category together with a monoidal $\bo{CMon}$-adjunction
	\begin{center}
		
		\begin{tikzpicture}[baseline=(current  bounding  box.south), scale=2]

			\node (f) at (0,0.4) {$\bar{\W}$};
			\node (g) at (1.3,0.4) {$\bar{\V}$};
			\node (h) at (0.66, 0.48) {$\perp$};
			\path[font=\scriptsize]

			([yshift=-1.3pt]f.east) edge [->] node [below] {$U$} ([yshift=-1.3pt]g.west)
			([yshift=2pt]f.east) edge [bend left,<-] node [above] {$F$} ([yshift=2pt]g.west);
		\end{tikzpicture}
		
	\end{center}
	with $U$ a conservative, filtered colimit and regular epimorphism preserving $\V$-functor for which the comparison maps $ UA\otimes UB\to U(A\otimes B) $
	are regular epimorphisms. Then $\W$ is locally dualizable.
\end{prop}

\begin{obs}\label{transpose-dualizable}
	As a final observation before the next step take two $\V$-functors $ M \colon \C^{op}\to\V$ and $H\colon \C\to\V$, where $\C$ has $\G$-copowers. Consider any $P\in\G$, $C\in\C$, and maps $x\colon P\to  M  C$ and $y\colon P^*\to HC$. Since $P\cong I\otimes P$ we can consider the transposes of $x$ to get a map $I\to [P, M  C]$ which, since $ M $ preserves such powers (being absolute), corresponds to a map $x'\colon I\to  M (P\cdot C)$. Similarly, $y$ corresponds to a morphism $I\to [P^*,HC]\cong P\otimes HC$ which, since $H$ preserves such copowers, is the same as a map $y'\colon I\to H(P\cdot C)$. Then it is easy to see that the square below commutes.
	
	\begin{center}
		\begin{tikzpicture}[baseline=(current  bounding  box.south), scale=2]
			
			\node (a) at (0,0.8) {$I$};
			\node (b) at (-1,0) {$ M (C)\otimes HC$};
			\node (c) at (1,0) {$ M (P\cdot C)\otimes H(P\cdot C)$};
			\node (d) at (0,-0.8) {$ M * H$};

			\path[font=\scriptsize]
			
			(a) edge [->] node [left] {$x\stackrel{\scriptscriptstyle{P}}{\otimes}  y\ \ $} (b)
			(a) edge [->] node [right] {$\ \ x'\otimes y'$} (c)
			(b) edge [->] node [left] {$q_C\ \ $} (d)
			(c) edge [->] node [right] {$\ \ q_{P\cdot C}$} (d);
			
		\end{tikzpicture}
	\end{center}
\end{obs}

\subsection{The characterization theorem}\label{theorem-locallydualiz}

We are now ready to prove our results starting with an adapted version of Lemma~\ref{goodunit}. For the remainder of this section $\V$ is assumed to be locally dualizable.

\begin{lema}\label{multigoounit}
	Let $\C$ be a $\V$-category with finite direct sums and $\G$-copowers. Let $ M \colon \C^{op}\to\V$ and $H\colon \C\to\V$ be two $\V$-functors. Then for each arrow $x\colon I\to   M *H$ there exist $C\in \C$, $y\colon I\to M  C$, and $z\colon I\to HC$ for which the triangle
	\begin{center}
		\begin{tikzpicture}[baseline=(current  bounding  box.south), scale=2]

			\node (b0) at (1.3,0.7) {$ M  C\otimes HC$};
			\node (c0) at (0,0.7) {$I$};
			\node (d0) at (1.3,0) {$ M *H$};
			
			\path[font=\scriptsize]
			
			(c0) edge [dashed, ->] node [above] {$ y\otimes z$} (b0)
			(b0) edge [->] node [right] {$ \rho_C$} (d0)
			(c0) edge [->] node [below] {$x\ $} (d0);
			
		\end{tikzpicture}	
	\end{center}
	commutes, where the vertical map is taken from the colimiting cocone defining $ M *H$.
\end{lema}
\begin{proof}
	The weighted colimit $ M * H$ can be seen as a coend; hence we have a regular epimorphism in $\V_0$
	$$ \sum_{C\in \C} M  C\otimes HC\stackrel{\rho}{\longrightarrow} \int^{C\in\C} M  C\otimes HC\cong  M *H. $$
	Since $\bar{\V}(I,-)$ preserves regular epimorphisms and coproducts it follows that $x$ factors through an element $h\in \textstyle\sum_{C\in \C}\bar{\V}(I, M  C\otimes HC)$ in $\bo{CMon}$. Such an element is given by $h=\textstyle\sum_{i=1}^n h_i$ with $h\colon I\to  M  D_i\otimes HD_i$ (by definition of coproduct in $\bo{CMon}$). By the hypotheses on $\V$ we can write each $h_i$ as a finite direct sum of elements of the form $y'_j\stackrel{\scriptscriptstyle{P_j}}{\otimes}  z'_j$, for $y'_j\colon P_j\to  M (D_i)$ and $z'_j\colon P_j^*\to H(D_i)$. In other words, relabelling the objects and the morphisms, we can write 
	$h=\textstyle\sum_{j=1}^n (y'_j\stackrel{\scriptscriptstyle{P_j}}{\otimes}  z'_j)$ for some $P_j\in\P$, $y'_j\colon P_j\to  M (D_j)$, and $z'_j\colon P_j^*\to H(D_j)$.\\
	Remember now that powers by $P_i$ and $P_i^*$ are absolute, it then follows that each $y'_j$ corresponds by transposition to $y_j\colon I\to  M (P_j\cdot D_j)$, since $ M $ is contravariant, and each $z'_j$ corresponds to $z_j\colon I\to H(P_j\cdot D_j)$. Therefore
	$$ x=\rho(h)=\rho(\sum_{j=1}^n y'_j\stackrel{\scriptscriptstyle{P_j}}{\otimes}  z'_j) =\sum_{j=1}^n \rho(y'_j\stackrel{\scriptscriptstyle{P_j}}{\otimes}  z'_j)=\sum_{j=1}^n \rho(y_j\otimes z_j)$$
	as an element of $\bar{\V}(I, M *H)$, where the last equality holds thanks to Remark~\ref{transpose-dualizable}.
	Consider then the object of $\C$ given by
	$$ C:=\bigoplus_{j=1}^n  P_j\cdot D_j;$$ 
	we can define $y:=\textstyle\sum_{j=1}^n y_j\colon I\to M (C)$ and $z:=\textstyle\sum_{j=1}^n z_j\colon I\to H(C)$ so that by construction we have 
	$$\rho(y\otimes z)=\sum_{j=1}^n \rho(y_j\otimes z_j)=x$$ as desired.
\end{proof}

With the same hypotheses as in the Lemma~above we can reduce flat colimits to filtered colimits, giving (under assumptions on $\C$) a proof of condition (I) from the Introduction.

\begin{cor}
	Let $\C$ be a $\V$-category with finite direct sums and $\G$-copowers. Let $ M \colon \C^{op}\to\V$ be a $\V$-functor for which $ M *-\colon [\C,\V]\to \V$ preserves $\alpha$-small conical limits of representables, and define $ M _I:=\V_0(I, M _0-)\colon \C_0^{op}\longrightarrow\bo{Set}$ as usual.  
	Then the ordinary category $\tx{El}( M _I)$ is $\alpha$-filtered, and so $M_I$ is $\alpha$-flat.
\end{cor}
\begin{proof}
	Using the Lemma~above the proof is exactly the same as that of Corollary~\ref{filtered-elements}.
\end{proof}

As a consequence we obtain the following proposition, showing that (II) also holds under these conditions.

\begin{prop}\label{flat+absolute=filtered}
	Let $\C$ be a $\V$-category with finite direct sums and $\G$-copowers, and let $ M \colon \C^{op}\to\V$ be a $\V$-functor. The following are equivalent:\begin{enumerate}\setlength\itemsep{0.25em}
		\item $ M $ is $\alpha$-flat;
		\item $ M *-\colon [\C,\V]\to \V$ preserves $\alpha$-small conical limits of representables;
		\item $\tx{El}( M _I)$ is $\alpha$-filtered;
		\item $ M $ is an $\alpha$-filtered colimit of representables.
	\end{enumerate}
	Moreover, in this case the canonical map:
	$$ \tx{colim} \left(\tx{El}( M _I)_{\V} \stackrel{\pi_{\V}}{\longrightarrow} \C \stackrel{Y}{\longrightarrow} [\C^{op},\V]\right) \longrightarrow  M $$
	is invertible.
\end{prop}
\begin{proof}
	$(1)\Rightarrow (2)$ is trivial, $(2)\Rightarrow (3)$ is true by the Corollary~above, while $(4)\Rightarrow (1)$ follows from the fact that representables $\V$-functors are $\alpha$-flat and the $\alpha$-flat $\V$-functors are closed under $\alpha$-filtered colimits in $[\C^{op},\V]$.
	
	It only remains to show  $(3)\Rightarrow (4)$. Since the category of elements is $\alpha$-filtered, we only need to prove that the isomorphism in the statement holds. Colimits are computed pointwise in $[\C^{op},\V]$; thus it's enough to show that the canonical map 
	$$ \tx{colim} \left(\tx{El}( M _I)_{\V} \stackrel{\pi_{\V}}{\xrightarrow{\hspace*{0.6cm}}} \C \stackrel{\C(C,-)}{\xrightarrow{\hspace*{1cm}}} \V\right)\longrightarrow  M (C) $$
	is invertible for any $C\in\C$. For that consider the strong generator $\G$; we know that $\G$-powers are absolute and that $\G$ consists of finitely presentable objects in $\V_0$. Moreover the ordinary functor  $ M _I=\V_0(I, M _0-)$ is the colimit of the corresponding ordinary diagram based on $\tx{El}( M _I)$. Thus for each $P\in\G$ we can write
	\begin{equation*}
		\begin{split}
			\V_0(P, M (C))&\cong \V_0(I,[P, M (C)])\\
			&\cong \V_0(I, M (P\cdot C))\\
			&\cong  M _I(P\cdot C)\\
			&\cong \underset{(x,D)\in \tx{El}( M _I)}{\tx{colim}}\ \C_0(P\cdot C,D)\\
			&\cong \underset{(x,D)\in \tx{El}( M _I)}{\tx{colim}}\ \V_0(I,\C(P\cdot C,D))\\
			&\cong \underset{(x,D)\in \tx{El}( M _I)}{\tx{colim}}\ \V_0(I,[P,\C(C,D)])\\
			&\cong \underset{(x,D)\in \tx{El}( M _I)}{\tx{colim}}\ \V_0(P,\C(C,D))\\
			&\cong \V_0(P,\underset{(x,D)\in \tx{El}( M _I)}{\tx{colim}}\ \C(C,D)).\\
		\end{split}
	\end{equation*}
	As a consequence, since the family $\G$ is strongly generating, it follows at once that $ M (C)\cong \tx{colim}\ \C(C,\pi_\V-)$ as desired.
\end{proof}

\begin{obs}
	When $\alpha=\aleph_0$ it's enough to ask that $ M *-$ preserve equalizers, since it already preserves finite products (being direct sums). Moreover, for  $\V=\bo{Ab}$ and $\alpha=\aleph_0$ this is Theorem~3.2 of \cite{ObRo70:articolo}.
\end{obs}

Therefore we obtain a characterization of $\alpha$-flat $\V$-functors in general:

\begin{prop}\label{flat-ab}
	Let $ M \colon \C^{op}\to\V$ be a $\V$-functor; the following are equivalent:\begin{enumerate}\setlength\itemsep{0.25em}
		\item $ M $ is $\alpha$-flat, or equivalently $ M *-\colon [\C,\V]\to \V$ preserves $\alpha$-small limits;
		\item $ M *-\colon [\C,\V]\to \V$ preserves $\alpha$-small limits of representables;
		\item $ M $ lies in the closure of the representables under $\G$-copowers, finite direct sums, and $\alpha$-filtered colimits.
	\end{enumerate}
\end{prop}
\begin{proof}
	$(1)\Rightarrow (2)$ is trivial, while $(3)\Rightarrow (1)$ follows from the fact that representables functors are $\alpha$-flat and these are closed under $\alpha$-filtered and absolute colimits in $[\C,\V]$.
	
	$(2)\Rightarrow (3)$. Let $ M \colon \C^{op}\to\V$ be as in $(2)$ and $J\colon \C\hookrightarrow\D$ be the inclusion of $\C$ into its free completion under finite direct sums and $\G$-copowers. Consider the weight $ M ':=\tx{Lan}_{J^{op}} M $, this still has the property that $ M '*-\colon [\D,\V]\to \V$ preserves $\alpha$-small conical limits of representables: since finite direct sums  and $\G$-copowers are $\alpha$-small, the limit of any $\alpha$-small diagram landing in $\D^{op}$ can be rewritten as the weighted limit of one landing in $\C^{op}$; thus $ M '*-$ preserves $\alpha$-small conical limits of representables by condition $(2)$. Now the domain of $ M '\colon \D^{op}\to\V$ satisfies the conditions of Proposition~\ref{flat+absolute=filtered}, therefore we can write $ M '\cong\colim (YK)$ with $K\colon \B_\V\to\D$ an $\alpha$-filtered diagram, and $Y\colon \D\to[\D^{op},\V]$ the Yoneda embedding. Then one concludes since each element of $\D$ is a finite direct sum of $\G$-copowers of objects from $\C$.
\end{proof}

\begin{ese}$ $
	\begin{itemize}\setlength\itemsep{0.25em}
		\item For $\V=\bo{Ab}$ and $\alpha=\aleph_0$ we recover the characterization of flat additive functors from \cite[Theorem~3.2]{ObRo70:articolo}.
		\item When $\V=\bo{GAb}$ we obtain that a $\V$-functor is flat if and only if it is a filtered colimit of suspensions of finite direct sums of representables.
		\item When $\V=R\tx{-}\bo{Mod}$, $\alpha=\aleph_0$, and $\C=\I$ we recover Lazard's criterion \cite{Laz64:articolo}: an $R$-module is flat if and only if it is a filtered colimit of free modules.
	\end{itemize}
\end{ese}

This allows us to characterize $\alpha$-flat colimits in terms of absolute and $\alpha$-filtered ones.

\begin{teo}\label{con-abs=flat}
	A $\V$-category $\A$ has $\alpha$-flat colimits if and only if it has finite direct sums, $\G$-copowers and $\alpha$-filtered colimits. A $\V$-functor from such an $\A$ preserves $\alpha$-flat colimits if and only if it preserves $\alpha$-filtered colimits.
\end{teo}
\begin{proof}
	Since $\alpha$-filtered colimits and absolute colimits are $\alpha$-flat, any $\V$-category $\A$ with $\alpha$-flat colimits has all of them.
	
	Vice versa, assume that $\A$ has the colimits above and consider an $\alpha$-flat weight $ M \colon \C^{op}\to\V$ together with a diagram $H\colon \C\to\A$ in $\A$. Let $J\colon \C\hookrightarrow\D$ be the inclusion of $\C$ into its free cocompletion under finite direct sums and $\G$-copowers. Since $\A$ has these colimits we can consider $H':=\tx{Lan}_JH$, while on the weighted side we take $ M ':=\tx{Lan}_{J^{op}} M $. By Lemma~\ref{flat-restriction} the weight $ M '$ is still $\alpha$-flat and, by construction, its domain satisfies the hypotheses of Proposition~\ref{flat+absolute=filtered}. Thus we can write $ M '\cong\tx{colim}YF$ as an $\alpha$-filtered colimit of representables; here $Y\colon \D\to[\D^{op},\V]$ is Yoneda and $F\colon \B_{\V}\to\D$ is a functor with $\alpha$-filtered domain.
	As a consequence $ M *H$ exists if and only if $ M '*H'$ exists, and so if and only if $\tx{colim} (H'F)$ exists (see the isomorphisms in the proof of \ref{flat-preservation}) and they coincide. 
	Thus the existence and preservation of the $\alpha$-flat colimit $ M *H$ is equivalent to the existence and preservation of the $\alpha$-filtered colimit $\tx{colim} H'F$.
\end{proof}

Therefore:

\begin{teo}\label{acc=conacc+cauchy}
	A $\V$-category $\A$ is $\alpha$-accessible if and only if it has finite direct sums and $\G$-copowers, and is conically $\alpha$-accessible.
\end{teo}
\begin{proof}
	By Theorem~\ref{con-abs=flat} above an object $A$ of $\A$ is $\alpha$-presentable if and only if is conically $\alpha$-presentable, so that $\A_{\alpha}=\A_\alpha^c$. Arguing as above, for any $\alpha$-flat $ M \colon \C^{op}\to\V$ and diagram $H\colon \C\to\A_{\alpha}\subseteq\A$, the colimit $ M *H$ can be replaced by an $\alpha$-filtered one $\tx{colim} (H'F)$, where $H'$ is the left Kan extension of 
	$H$ along the free cocompletion $\D$ of $\C$ under finite direct sums and $\G$-copowers, and $F\colon \B_\V\to\D$ has $\alpha$-filtered domain. Then $H'F\colon \B_\V\to\A$ still lands in $\A_\alpha$ since this is closed in $\A$ under finite direct sums and $\G$-copowers. Thus an object of $\A$ is an $\alpha$-flat colimit of $\alpha$-presentables if and only if it is an $\alpha$-filtered colimit of (conically) $\alpha$-presentables. The result then follows.
\end{proof}

\begin{obs}
	For $\V=\bo{Ab}$ and $\alpha=\aleph_0$ see Example~9.2 of \cite{BQR98}.
\end{obs}

Once more, using Theorem~\ref{acc=conacc+cauchy} and the results of \cite[Section~7]{BQR98}, we can compare conical accessible $\V$-categories and models of $\V$-sketches as follows: 

\begin{teo}
	Let $\A$ be a $\V$-category; the following are equivalent:\begin{enumerate}
		\item $\A$ is accessible;
		\item $\A$ has finite direct sums and $\G$-copowers, and is conically accessible;
		\item $\A$ is equivalent to the $\V$-category of models of a $\V$-sketch.
	\end{enumerate}
\end{teo}

As a direct consequence we characterize the Cauchy complete $\V$-categories.

\begin{cor}\label{Cauchy}
	Let $\C$ be a $\V$-category; the following are equivalent:
	\begin{enumerate}\setlength\itemsep{0.25em}
		\item $\C$ is Cauchy complete;
		\item $\C$ has finite direct sums, copowers (and hence powers) by dualizable objects, and splitting idempotents;
		\item $\C$ has finite direct sums, $\G$-copowers, and splitting of idempotents.
	\end{enumerate}
\end{cor}
\begin{proof}
	$(1)\Rightarrow (2)\Rightarrow (3)$ are trivial, and $(3)\Rightarrow (1)$ is a consequence of Theorem~\ref{con-abs=flat}, using the fact that an absolute weight is one that is $\alpha$-flat for any $\alpha$.\\
\end{proof}

For the next result consider the set $\langle\G\rangle$ given by the closure of $\G\cup\{I\}$ under tensor product.

\begin{prop}
	Let $\C$ be any small $\V$-category. A weight $ M \colon \C^{op}\to\V$ is absolute if and only if there exist objects $C_1,\dots,C_n\in \C$ and $P_1,\dots. P_n\in\langle\G\rangle$ such that $ M $ is a split subobject of $$\bigoplus_{i=1}^{n}P_i\cdot \C(-,C_i).$$
\end{prop}
\begin{proof}
	If the latter holds then $ M $ is an absolute colimit of absolute weights; thus it is absolute itself.
	
	Assume now that $ M $ is absolute. Let $\D$ be the full subcategory of $ [\C^{op},\V]$ spanned by the finite direct sums of $\langle\G\rangle$-copowers of representables;
	denote by $J\colon \C\hookrightarrow\D$ and note that $\D$ is closed in $[\C^{op},\V]$ under finite direct sums and $\G$-copowers. Consider now $ M ':=\tx{Lan}_{J^{op}}( M )$; then by Lemma~\ref{flat-restriction} the weight $ M '$ is still absolute and, by construction, its domain satisfies the hypotheses of Proposition~\ref{flat+absolute=filtered}. It follows that $ M '$, seen in $[\D^{op},\V]_0$ is an ordinary absolute colimit of representables; therefore $ M '$ is a split subobject of $\D(-,D)$ for some $D\in\D$. 
	As a consequence $ M \cong M '\circ J^{op}$ is a split subobject of $\D(J-,D)$. Now, by construction of $\D$, the object $D$ can be written as $\textstyle\sum_{i=1}^{n}(P_i\cdot JC_i)$ for some $C_i\in\C$ and $P_i\in\langle\G\rangle$. Therefore $\D(J-,D)\cong \textstyle\sum_{i=1}^{n}P_i\cdot\C(-,C_i) $ and the result follows.
\end{proof}

\begin{obs}
	If we consider $\V=\bo{GAb}$ we recover the results of section 6 from \cite{NST2020cauchy}. In fact we can consider $\G=\{S^nI\}_{n\in\mathbb{Z}}$ given by the suspensions of the unit. Then the proposition above is \cite[Proposition~6.1]{NST2020cauchy} and Corollary~\ref{Cauchy} is \cite[Proposition~6.2]{NST2020cauchy}.
\end{obs}

\section{Flat does not equal filtered plus absolute in general} \label{counterexample}

In this section we consider the base of enrichment to be the cartesian closed category $\V=\bo{Set}^G$ of $G$-sets, for a non-trivial finite group $G$. We will prove that in this case the flat $\V$-functors do not lie in the closure of representable functors under absolute and filtered colimits.

First of all note that we have an adjunction
\begin{center}
	
	\begin{tikzpicture}[baseline=(current  bounding  box.south), scale=2]

		\node (f) at (0,0.4) {$\bo{Set}^G$};
		\node (g) at (1.2,0.4) {$\bo{Set}$};
		\node (h) at (0.63, 0.45) {$\perp$};
		\path[font=\scriptsize]

		([yshift=-1.3pt]f.east) edge [->] node [below] {$U$} ([yshift=-1.3pt]g.west)
		([yshift=2pt]f.east) edge [bend left,<-] node [above] {$F$} ([yshift=2pt]g.west);
	\end{tikzpicture}
	
\end{center} 
where $U=\bo{Set}^G(G,-)$ takes the underlying sets, $F=G\times -$ sends a set $A$ to the $G$-set $G\times A$ with the free action, and we are denoting with $G\in\bo{Set}^G$ also the representable functor corresponding to the only object of the group $G$. Note that $U$ is conservative, continuous, cocontinuous, strong monoidal, and strong closed. 

The object $G$ is finitely presentable and a strong generator for $\bo{Set}^G$; moreover the functors $\bo{Set}^G(1,-)$   and $\bo{Set}^G(G\times G,-)\cong \bo{Set}(G,U-)$ are finitary (since $G$ is finite). Therefore $\bo{Set}^G$ is locally finitely presentable as a closed category.

Since $U$ is strong monoidal it follows that there is an induced 2-functor:
$$ U_*\colon \V\tx{-}\bo{Cat}\longrightarrow \bo{Cat}$$
and a $\V$-functor $\hat{U}=(U_*\V)(1,-)\colon U_*\V\to\bo{Set}$ which acts as $U$ on objects. Now, for each $\V$-weight $ M \colon \C^{op}\to\V$ we can define
$$  M _U\colon  U_*\C^{op}\stackrel{U_* M }{\xrightarrow{\hspace*{0.6cm}}}U_*\V \stackrel{\hat{U}}{\xrightarrow{\hspace*{0.4cm}}}\bo{Set}. $$
Since $U$ is moreover cocontinuous it follows that for any $\V$-category $\C$, any $M\colon\C^{op}\to\V$ and $H\colon\C\to\V$, we have an isomorphism 
$$U(M*H)\cong M_U*H_U$$
where the colimit on the right is an ordinary weighted colimit. See also Section~\ref{DG-setting} for related properties about change of base.

Note that to give a $\V$-category $\C$ is equivalently to give an ordinary category $\C$ whose homs are endowed with group actions $G\times \C(A,B)\to \C(A,B)$ for which the identities are fixed points and the composition maps are equivariant. It follows that $U_*\C$ is the same as $\C$ with the only difference being that the group actions on the homs are forgotten. The category $U_*\C$ should not be mistaken with the underlying category $\C_0$ of $\C$ which has homs  $\C_0(A,B)=\tx{Fix}\ \C(A,B)$ given by the fixed points of the action on $\C(A,B)$.

Let's start with a result comparing enriched and ordinary flatness which can be seen as a consequence of \cite[Theorem~18]{BR08}.

\begin{prop}\label{G-setsflatness}
	Let $\C$ be a $\V$-category and $M\colon\C^{op}\to\V$ be a $\V$-functor. Then $M$ is $\alpha$-flat if and only if $M_U$ is $\alpha$-flat. 
\end{prop}
\begin{proof}
	Assume first that $M$ is $\alpha$-flat; it's enough to prove that $\tx{El}(M_U)$ is $\alpha$-filtered. Note that, since $U=\bo{Set}^G(G,-)$, the category $\tx{El}(M_U)$ can be described as in Remark~\ref{X-elements-0} with $G$ in place of $X$. Using that and the fact that $\bo{Set}^G(G,-)$ is cocontinuous and strong monoidal, one can easily adapt the proofs of Lemma~\ref{goodunit} and Corollary~\ref{filtered-elements} to show that 
	$\tx{El}(M_U)$ is $\alpha$-filtered (just replace with $G$ all instances of $I$ in the proofs).
	
	Conversely assume that $M _U$ is $\alpha$-flat; we need to show that $M *-\colon [\C,\V]\to \V$ preserves all $\alpha$-small conical limits and powers by $G$. Since $U$ is continuous and conservative, $M *-$ preserves $\alpha$-small conical limits if and only if $U( M*-)\colon [\C,\V]_0\to \bo{Set}$ preserves them. But $U( M*-)\cong  M_U*(-)_U$, where $(-)_U$ is continuous (since $U$ is) and $ M_U*-$ is $\alpha$-continuous because $ M _U$ is $\alpha$-flat. Thus we are left to prove that $M *-$ preserves powers by $G$. Let $H\colon\C\to\V$ be any $\V$-functor; then the comparison map  $M*(G\pitchfork H)\to G\pitchfork(M*H)$ is invertible if and only if its image under $U$ is so. Therefore 
	\begin{align}
		U(M*(G\pitchfork H)) &\cong M_U* (G\pitchfork H)_U\nonumber \\
		&\cong M_U* \bo{Set}(UG,H_U-)\tag{1} \\
		&\cong \bo{Set}(UG,M_U* H_U)\tag{2} \\
		&\cong \bo{Set}(UG,U(M*H))\nonumber \\
		&\cong U(G\pitchfork (M*H))\tag{3}
	\end{align}
	where we used that $(G\pitchfork H)_U\cong \bo{Set}(UG,H_U-)$ for $(1)$, that $M_U$ is $\alpha$-flat and $G$ is finite for $(2)$, and that $U$ is strong closed for $(3)$. It follows that $M$ is $\alpha$-flat.
\end{proof}

We are now ready to provide an explicit example of a $\V$-category for which flat presheaves on $\C$ are not in the closure of the representables under absolute and filtered colimit.

Define $\C$ as follows: the objects $\tx{Ob}(\C)=\mathbb{N}$ are natural numbers; for each $n,m\in\mathbb{N}$ we set $\C(n,m)=\emptyset$ if $n>m$, while $\C(n,m)=\{1_n\}$ consists only of the identity map (with trivial action) if $n=m$, and $\C(n,m)=G=F1$ (with action given by multiplication) if $n<m$. Composition $-\circ -\colon \C(m,l)\times\C(n,m)\to\C(n,l)$ is non-trivial only when $n<m<l$ and in that case is given by $g\circ h:=g$, for any $g,h\in G$. It's now easy to see that the composition maps are equivariant and well defined, and that the identities are fixed points of the action; therefore we obtain a $\V$-category $\C$. 

\begin{prop}
	The $\V$-category $\C$ has absolute colimits and filtered colimits, and the Yoneda embedding $Y\colon\C\to[\C^{op},\V]$ preserves them. 
\end{prop}
\begin{proof}
	To begin with note that the underlying category $\C_0$ of $\C$ is the discrete category $\mathbb{N}$; this is because $\C_0(n,n)=\{1_n\}$ and $\C_0(n,m)=\tx{Fix}\ \C(n,m)=\emptyset$ for any $n\neq m$ (here we are using the fact that the group $G$ is non trivial). As a consequence the only filtered diagrams that exist in $\C$ are the constant ones, and these have as colimit the same object they pick; these colimits are clearly preserved by $Y$. (Thus in fact the filtered colimits in $\C$ are absolute.)
	
	To conclude it's enough to show that every Cauchy $\V$-functor $M\colon\C^{op}\to\V$ is representable. Let then $M$ be Cauchy; by Proposition~\ref{G-setsflatness} the ordinary functor $M_U\colon U_*\C^{op}\to\bo{Set}$ is Cauchy as well (as usual use that a weight is Cauchy if and only if it is $\alpha$-flat for every $\alpha$). The category $U_*\C$ is Cauchy complete in the ordinary sense (since the only idempotents are the identities); therefore there exists $n\in U_*\C$ such that $M_U\cong (U_*\C)(-,n)$. We wish to prove that actually $M\cong \C(-,n)$. For that, let $m\in\C$ be any other object and 
	$$ M_{n,m}\colon\C(m,n)\longrightarrow [Mn,Mm] $$
	be the action of $M$ on morphisms in $\bo{Set}^G$. Since $UMn\cong M_Un\cong (U_*\C)(n,n)=\{1_n\}$, it follows that $Mn=1$ is the terminal object in $\bo{Set}^G$. As a consequence the maps $M_{n,m}$ are actually of the form
	$$ M_{n,m}\colon\C(m,n)\longrightarrow Mm $$
	and define a $\V$-natural transformation $\C(-,n)\to M$. Since $M_U\cong (U_*\C)(-,n)$ the maps $U( M_{n,m})$ are bijections of sets and hence, since $U$ is conservative, the $\V$-natural transformation $M_{n,m}$ is an isomorphism. This proves that $M$ is representable.
\end{proof}

\begin{teo}
	The terminal object $\Delta 1$ in $[\C^{op},\V]$ is flat but doesn't lie in the closure of the representables under absolute and filtered colimits.
\end{teo}
\begin{proof}
	By the proposition above the closure of $\C$ in $[\C^{op},\V]$ under absolute and filtered colimits is $\C$ itself; therefore it's enough to prove that $\Delta 1$ is flat but not representable.
	
	By Proposition~\ref{G-setsflatness}, the $\V$-functor $\Delta 1$ is flat if and only if the functor $\Delta 1\colon U_*\C^{op}\to\bo{Set}$ is; and the latter is flat since its category of elements is equal to $U_*\C$, which is filtered. Finally, $\Delta 1\colon\C^{op}\to\V$ is not representable since $\C$ doesn't have a terminal object.
\end{proof}

\begin{cor}\label{acc-not-conacc+absolute}
	The $\V$-category $\C$ is Cauchy complete and conically finitely accessible, but doesn't have all flat colimits. In particular $\C$ is not finitely accessible. 
\end{cor}
\begin{proof}
	Filtered colimits are trivial in $\C$; therefore every object is conically finitely presentable. Since $\C$ is small this is enough to imply that it is conically finitely accessible. 
	
	By the theorem above, the terminal $\V$-functor $\Delta 1\colon\C^{op}\to\V$ is flat; thus to conclude it's enough to show that the colimit $\Delta 1*\tx{id}_\C$ doesn't exist in $\C$. In order to obtain a contradiction assume that  $n\cong \Delta 1*\tx{id}_\C$ exists; then 
	$$ \C(n,n+1)\cong [\C^{op},\V](\Delta 1,\C(-,n+1)) $$
	by the universal property of the colimit. But $ \C(n,n+1)$ is not empty, while the underlying set of $[\C^{op},\V](\Delta 1,\C(-,n+1))$ is empty since there are no maps $1\to \C(n+2,n+1)$. Therefore $\Delta 1*\tx{id}_\C$ doesn't exist.
\end{proof}

\section{When flat equals protofiltered plus absolute}\label{DG}

We begin this section by introducing the bases we are interested in, the main example to keep in mind being the symmetric monoidal closed category $\bo{DGAb}$ of chain complexes. In Section~\ref{DGtheorem} we show that in this context the $\alpha$-flat colimits are generated by the absolute colimits together with what we call {\em $\alpha$-protofiltered colimits}. These include, but may not reduce to, the usual $\alpha$-filtered colimits. In Section~\ref{appendix} we prove that $\alpha$-protofiltered colimits are genuine $\alpha$-flat colimits.

\subsection{Setting}\label{DG-setting}

Let $\V=(\V_0,\otimes,I)$ and $\W=(\W_0,\otimes, J)$ be symmetric monoidal closed complete and cocomplete categories for which $\V$ is locally dualizable with strong generator $\G\ni I$, and $\W$ has finite direct sums. Moreover we assume that there is a functor $ U\colon\W_0\to\V_0 $ with adjoints $L\dashv U\dashv R$ such that:\begin{enumerate}
	\item[(a)] $U$ is conservative, strong monoidal, and strong closed;
	\item[(b)] $UL\G$ is still a strong generator of $\V$;
	\item[(c)] for any $X\in\G$ the objects $ULX$ is still dualizable.  
\end{enumerate}  

\begin{obs}
	It's useful to note the following properties:\begin{enumerate}
		\item[(i)] $U$ is conservative if and only if $L\G$ is a strong generator.
		\item[(ii)] Given $(a)$, similarly $R$ is conservative if and only if $UL\G$ is a strong generator.
		\item[(iii)] Given $(a)$, an object $Y\in\W$ is dualizable if and only if $UY$ is so, thus $(c)$ is equivalent to the objects of $L\G$ being dualizable. And if $UL\cong ULI\otimes-$ then this is equivalent to $ULI$ (or $LI$) being dualizable.
	\end{enumerate}
\end{obs}

\begin{ese}$ $
	\begin{itemize}
		\item Let $\W=\bo{DGAb}$, $\V=\bo{GAb}$, and $U$ be the forgetful functor. Then $U$ has both adjoints and $LI$ is the chain complex having $(0)$ in all the degrees but $(LI)_0=(LI)_{-1}=\mathbb{Z}$, with differential $d=\tx{id}$ between them. Moreover $U$ is conservative, strong monoidal, and strong closed \cite[Section~6]{EK66closed}. Let now $\G$ be the strong generator of $\bo{GAb}$ consisting of the dualizable objects $S_nI$, for $n\in\mathbb{Z}$. Then $UL\G=\{S_nI\oplus S_{n-1}I\}_{n\in\mathbb{Z}}$ is still a strong generator of $\bo{GAb}$ made of dualizable objects. It follows that $\bo{DGAb}$ is an example of such base of enrichment.
		
		\item Let $\V$ be locally dualizable and $H$ be a cocommutative Hopf algebra in $\V$ which is dualizable as an object of $\V$. We can consider $\W$ to be the symmetric monoidal closed category of $H$-modules with $U\colon\W\to\V$ the forgetful functor. Then $U$ has both adjoints, satisfies $UL\cong H\otimes -$, and $(a)$ holds. If $\G$ is a strong generator made of dualizable objects for $\V$, then the elements of $UL\G$ are of the form $H\otimes X$, for $X\in\G$. Then $(c)$ holds by the remark above since $H$ is dualizable, moreover $UL\G$ is still a strong generator since for every $X\in\G$ we have a split epimorphism $\epsilon\otimes 1\colon H\otimes X\to X$, where $\epsilon\colon H\to I$ is the counit of $H$. (The facts about Hopf algebras mentioned above can be found in Chapter 15 of \cite{street_2007}).
	\end{itemize}
	
\end{ese}

	The following are a consequence of $(a)$-$(c)$:\begin{itemize}
		\item $L\dashv U$ is an op-monoidal adjunction \cite[Theorem~1.2]{Kel-procedings} and $L(UY\otimes X)\cong A\otimes LX$;
		\item $U\dashv R$ is a monoidal adjunction \cite[Theorem~1.2]{Kel-procedings} and $R[UY,X]\cong [Y,RX]$;
		\item $U$ is monadic by Beck's monadicity theorem \cite[Theorem~4.4.4]{Bor94:libro};
		\item $T=UL$ is an op-monoidal Hopf monad \cite[Proposition~2.14]{BLV11}.
	\end{itemize} 
	In particular $LU\cong LI\otimes -$ and $RU\cong [LI,-]$.

Since $U$ is strong monoidal it follows that there is an induced 2-functor:
$$ U_*\colon \W\tx{-}\bo{Cat}\longrightarrow \V\tx{-}\bo{Cat}$$
and a $\V$-functor $\hat{U}=(U_*\W)(I,-)\colon U_*\W\to\V$ which acts as $U$ on objects. Now, for each $\W$-weight $ M \colon \C\to\W$ we can consider the composite
$$  M _U\colon  U_*\C\stackrel{U_* M }{\xrightarrow{\hspace*{0.7cm}}}U_*\W \stackrel{\hat{U}}{\xrightarrow{\hspace*{0.5cm}}}\V $$
as a $\V$-weight. Given a $\W$-category $\C$ we denote by $S_\C\colon \C_0\to(U_*\C)_0$ the identity-on-objects functor which acts by applying $U$ on morphisms: given a morphism $f\colon J\to\C(A,B)$ in $\C_0$ we define $S_\C(f):=Uf\colon I\cong UJ\to U\C(A,B)=(U_*\C)(A,B)$. 

The following lemmas are standard results about change of base along a monoidal functor which is continuous, cocontinuous, strong monoidal and strong closed.

\begin{lema}\label{U-copowers}
	Let $\C$ be a $\W$-category. Then:\begin{enumerate}
		\item for any ordinary $H\colon\E\to\C_0$, if the limit of $H$ exists in $\C$ then it is also the limit of $S_\C H$ in $U_*\C$;
		\item if the power $X\pitchfork A$ exists in $\C$ then it is also the power $UX\pitchfork A$ in $U_*\C$.
	\end{enumerate}
	The same property holds with the corresponding conical colimits and copowers.
\end{lema}
\begin{proof}
	$(1)$ Assume that $\lim H$ exists in $\C$; we need to prove that it is also the limit of $(S_\C H)$, if seen as an object of $U_*\C$. Let $C$ be any object of $U_*\C$; then
	\begin{equation*}
		\begin{split}
			(U_*\C)(C,\lim H) &\cong U(\lim\C(C,H-)_0)\\
			&\cong \lim U\circ \C(C,H-)_0\\
			&\cong \lim (U_*\C)(C,S_\C H-)_0\\
		\end{split}
	\end{equation*}
	in $\V_0$, where we used that $U$ is continuous and that $(U_*\C)(C,S_\C-)_0\cong U\circ \C(C,-)_0$. It follows that $\lim H$ is $\lim (S_\C H)$ in $U_*\C$.
	
	$(2)$ Assume that $X\pitchfork A$ exists in $\C$; then
	\begin{equation*}
		\begin{split}
			(U_*\C)(C,X\pitchfork A) &\cong U[X,\C(C,A)]\\
			&\cong [UX,U\C(C,A)]\\
			&\cong [UX,(U_*\C)(C,A)]\\
		\end{split}
	\end{equation*} 
	naturally in $C\in U_*\C$. It follows that $X\pitchfork A$, seen in $U_*\C$, is the power of $A$ by $UX$.
	
	The dual property involving colimits holds by the arguments above just replacing $\C$ with $\C^{op}$.
\end{proof}

\begin{lema}\label{underlyingGfunct}
	Let $\C$ be a $\W$-category. The ordinary functor 
	$$(-)_U\colon [\C,\W]_0\to[U_*\C,\V]_0$$ 
	is continuous; moreover $\C(C,-)_U\cong (U_*\C)(C,-)$.
\end{lema}
\begin{proof}
	To prove the first assertion note that $(-)_U$ can be written as the composite
	$$ [\C,\W]_0 \stackrel{S_{[\C,\W]}}{\xrightarrow{\hspace*{0.8cm}}} (U_*[\C,\W])_0 \stackrel{K}{\xrightarrow{\hspace*{0.5cm}}} [U_*\C,U_*\W]_0 \stackrel{\hat{U}\circ -}{\xrightarrow{\hspace*{0.8cm}}} [U_*\C,\V]_0$$
	where $K=(U_*)_{\C,\W}$ is the action of $U_*$ on homs. Now, $S_{[\C,\W]}$ is continuous by Lemma~\ref{U-copowers}, and $\hat{U}\circ -$ preserves all limits that exist in $[U_*\C,U_*\W]_0$ since $\hat{U}$ does so (being representable). Finally, observe that $K$ is fully faithful (since $U$ is continuous and strong closed) and $(U_*[\C,\W])_0$ contains the representables (since $K \circ U_*(Y_\C)\cong Y_{U_*\C}$); therefore $K$ preserves all limits that exists in $(U_*[\C,\W])_0$ (this is a general fact about full subcategories of presheaf categories which contain the representables). It follows at once that $(-)_U$ is continuous. Finally, that $\C(C,-)_U\cong (U_*\C)(C,-)$ is an immediate consequence of fact that $K \circ U_*(Y_\C)\cong Y_{U_*\C}$.
\end{proof}

\begin{lema}\label{DG-G-limits}
	Let $ M \colon \C^{op}\to\W$ be a $\W$-weight and $H\colon \C\to\W$ be a $\W$-functor. Then $ U( M *H)\cong  M _U* H_U$.
\end{lema}
\begin{proof}
	The weighted colimit $M *H$ can be seen as a coend
	\begin{center}
		\begin{tikzpicture}[baseline=(current  bounding  box.south), scale=2]

			\node (b) at (0,0) {$\underset{D,E\in\C}{\sum} \C(D,E)\otimes M D\otimes HE$};
			\node (c) at (2.5,0) {$\underset{C\in\C}{\sum} M C\otimes HC$};
			\node (a) at (4,0) {$\underset{}{M*H.}$};
			
			\path[font=\scriptsize]
			
			([yshift=2.5pt]c.east) edge [->>] node [above] {} ([yshift=2.5pt]a.west)
			([yshift=5pt]b.east) edge [->] node [above] {} ([yshift=5pt]c.west)
			([yshift=1pt]b.east) edge [->] node [below] {} ([yshift=1pt]c.west);
		\end{tikzpicture}
	\end{center}
	Since $U$ is continuous and strong monoidal, the image of this under it leads to the coequalizer
	\begin{center}
		\begin{tikzpicture}[baseline=(current  bounding  box.south), scale=2]

			\node (b) at (0,0) {$\underset{D,E\in U_*\C}{\sum} (U_*\C)(D,E)\otimes M_U D\otimes H_UE$};
			\node (c) at (3,0) {$\underset{C\in U_*\C}{\sum} M_U C\otimes M_UC$};
			\node (a) at (4.8,0) {$\underset{}{U(M*H),}$};
			
			\path[font=\scriptsize]
			
			([yshift=2.5pt]c.east) edge [->>] node [above] {} ([yshift=2.5pt]a.west)
			([yshift=5pt]b.east) edge [->] node [above] {} ([yshift=5pt]c.west)
			([yshift=1pt]b.east) edge [->] node [below] {} ([yshift=1pt]c.west);
		\end{tikzpicture}
	\end{center}
	where we used that $U\C(D,E)\cong (U_*\C)(D,E)$ and that $UMC\cong M_UC$, similarly for $H$. Since $M_U*H_U$ can be seen as the coend above, it follows that $U(M*H)\cong M_U*H_U$.
\end{proof}

Before moving on let us point out some properties of the base $\W$. First note that, even though $L\G$ is a strong generator of $\W$ made of dualizable objects, $\W$ may not be locally dualizable since the unit need not be projective.

Since the elements of $\G$ are finitely presentable and projective in $\V$, and $U$ is cocontinuous, the elements of $L\G$ are finitely presentable and projective as well; it follows that $\W$ is a finitary quasivariety \cite[Definition~4.4]{LT20:articolo}. Moreover $U$ sends a strong generator made of finitely presentable objects to one with the same property in $\V$. Finally note that for any $X,Y\in L\G$ the hom $\W_0(X\otimes Y,-)\cong \W_0(X,[Y,-]_0)$ preserves all colimits that $\W_0(X,-)$ preserves; therefore $\W$ is a symmetric monoidal closed finitary quasivariety and in particular locally finitely presentable as a closed category.

\subsection{The characterization theorem}\label{DGtheorem}

Fix $\W$, $\V$, and $U\colon\W\to\V$ as in the previous section.

\begin{obs}\label{powerswhocares}
	Note that the $\alpha$-small $\W$-weighted limits are generated by the $\alpha$-small conical ones and powers by the strong generator $L\G$. Since the latter are dualizable objects, and powers by them are absolute, it follows that a $\W$-functor preserves $\alpha$-small weighted limits of and only if it preserves $\alpha$-small conical ones. Therefore a $\W$-weight $ M $ is $\alpha$-flat if and only if $ M *-$ preserves all $\alpha$-small conical limits.
\end{obs}

Given $ M \colon \C^{op}\to\W$ we can consider two different categories of elements:\begin{itemize}\setlength\itemsep{0.25em}
	\item $\tx{El}( M _J)$: this has objects pairs $(C\in\C,x\colon J\to M (C))$ and arrows $f\colon (C,x)\to(D,y)$ given by $f\in\C_0(C,D)$ with $ M  f(y)=x$.
	\item $\tx{El}( M _{LI})$: objects are pairs $(C\in U_*\C,y\colon I\to M _U(C))$ and arrows $f\colon (C,x)\to(D,y)$ are given by $f\in (U_*\C)_0(C,D)$ with $ M  f(y)=x$.
	
\end{itemize}
Therefore we have an induced ordinary functor $S_ M \colon \tx{El}( M _J)\longrightarrow\tx{El}( M _{LI})  $
which makes the square below commute.
\begin{center}
	\begin{tikzpicture}[baseline=(current  bounding  box.south), scale=2]
		
		\node (a0) at (0,-0.8) {$\C_0$};
		\node (b0) at (1.1,-0.8) {$( U_*\C)_0$};
		\node (c0) at (0,0) {$\tx{El}( M _J)$};
		\node (d0) at (1.1,0) {$\tx{El}( M _{LI})$};
		
		\path[font=\scriptsize]
		
		(a0) edge [->] node [below] {$S_\C$} (b0)
		(c0) edge [->] node [left] {$\pi$} (a0)
		(d0) edge [->] node [right] {$\pi$} (b0)
		(c0) edge [->] node [above] {$S_ M $} (d0);
		
	\end{tikzpicture}	
\end{center}

\begin{prop}\label{underlyingleftadj}
	Let $\C$ be a $\W$-category with copowers by $LI$ and $ M \colon \C^{op}\to\W$ be a $\W$-weight. Then:\begin{enumerate}
		\item the functor $S_\C\colon \C_0\to(U_*\C)_0$ has a left adjoint $T_\C$ given by $T_\C C\colon =LI\cdot C$; 
		\item the functor $S_ M \colon \tx{El}( M _J)\to\tx{El}( M _{LI})$ has a left adjoint $T_ M $ which makes the square below commute.
	\end{enumerate}
	
	\begin{center}
		\begin{tikzpicture}[baseline=(current  bounding  box.south), scale=2]
			
			\node (a0) at (0,-0.8) {$\C_0$};
			\node (b0) at (1.1,-0.8) {$( U_*\C)_0$};
			\node (c0) at (0,0) {$\tx{El}( M _J)$};
			\node (d0) at (1.1,0) {$\tx{El}( M _{LI})$};
			
			\path[font=\scriptsize]
			
			(b0) edge [->] node [below] {$T_\C$} (a0)
			(c0) edge [->] node [left] {$\pi$} (a0)
			(d0) edge [->] node [right] {$\pi$} (b0)
			(d0) edge [->] node [above] {$T_ M $} (c0);
			
		\end{tikzpicture}	
	\end{center}
\end{prop}
\begin{proof}
	$(1)$. For any $C\in(U_*\C)_0$ we obtain:
	\begin{equation*}
		\begin{split}
			(U_*\C)_0(C,S_\C-) &\cong \V_0(I,(U_*\C)(C,S_\C-)_0)\\
			&\cong \V_0(I,U\circ \C(C,-)_0)\\
			&\cong \W_0(LI,\C(C,-)_0)\\
			&\cong \C_0(LI\cdot C,-)\\
		\end{split}
	\end{equation*}
	where we used the fact that $(U_*\C)(C,S_\C-)_0\cong U\circ \C(C,-)_0$. Therefore $(U_*\C)_0(C,S_\C-)$ is represented by the object $LI\cdot C$ of $\C_0$, and hence $LI\cdot(-)\colon (U_*\C)_0\to\C_0$ is a left adjoint to $S_\C$.
	
	$(2)$. By point $(1)$ applied to $\C=\W^{op}$ the functor $S_\W\colon \W_0\to (U_*\W)_0$ has a right adjoint given pointwise by $[LI,-]$; notice that to give an arrow $X\to Y$ in $(U_*\W)_0$ is the same as giving $UX\to UY$ in $\V_0$. Now, since $ M $ preserves powers by $LI$ (these being absolute), the left adjoint $T_\C$ to $S_\C$ extends to a left adjoint of $S_ M $ as follows: $T(C,y):=(T_\C C, y^t)$ where $T_\C C=LI\cdot C$ as above and $y^t\colon J\to  M (T_\C C)\cong [LI, M  C]$ is the transpose of $y\colon I\cong UJ\to \phi_U C\cong U M (C)$ seen as a map $J\to  M (C)$ in $(U_*\W)_0$. A routine verification shows that the resulting $T_ M $ is left adjoint to $S_ M $.
\end{proof}

In the presence of such a nice change of base we can reduce $\W$-flatness to $\V$-flatness:

\begin{prop}\label{DG-G-flat}
	Let $\C$ be a $\W$-category with finite direct sums and copowers by $L\G$; let $ M \colon \C^{op}\to\W$ be a $\W$-weight. The following are equivalent:\begin{enumerate}\setlength\itemsep{0.25em}
		\item $ M $ is $\alpha$-flat;
		\item $ M *-$ preserves conical $\alpha$-small limits of representables;
		\item $ M _U$ is $\alpha$-flat as a $\V$-weight;
	\end{enumerate} 
\end{prop}
\begin{proof}
	$(1)\Rightarrow (2)$ is trivial. Let's consider $(2)\Rightarrow (3)$. To begin with, note that by Lemma~\ref{U-copowers} the $\V$-category $U_*\C$ has finite direct sums and copowers by $UL\G$, which is a strong generator of $\V$ made of dualizable objects. Thus, by Proposition~\ref{flat+absolute=filtered}, it's enough to prove that the category $\tx{El}( M _{LI})=\tx{El}(( M _U)_I)$ is $\alpha$-filtered.
	
	Consider therefore an $\alpha$-small family of objects $(C_i,x_i)_{i\in I}$ in $\tx{El}( M _{LI})$; we need to find a cocone for that. Note that the family $(x_i)_i$ corresponds to a map $x\colon I\to \prod_i U M (C_i)$ in $\V$. Moreover
	\begin{equation*}
		\begin{split}
			\prod_i U M (C_i) &\cong U(\prod_i  M * \C(C_i,-))\\
			&\cong U(  M * \prod_i\C(C_i,-))\\
			&\cong  M _U* (\prod_i\C(C_i,-))_U\\
			&\cong  M _U* \prod_i U_*\C(C_i,-)\\
		\end{split}
	\end{equation*}
	As a consequence, $x$ corresponds to a map $x'\colon I\to  M _U* \prod_i U_*\C(C_i,-)$. By Lemma~\ref{multigoounit} there exist then $D\in U_*\C$, $y\colon I\to  M _U (D)$ and $f=(f_i)_i\in \prod_i (U_*\C)_0(C_i,D)$ which map down to $x'$. In other words we obtained $(D,y)$ and maps $f_i\colon (C_i,x,i)\to (D,y)$ in $\tx{El}( M _{LI})$, as desired.
	
	Consider now an $\alpha$-small family of parallel maps $\{f_i\colon (C,x)\to (D,y)\}_{i\in I}$ in $\tx{El}( M _{LI})$; we need to find an arrow coequalizing them. Since $S_ M $ has a left adjoint $T_ M $ we can consider the following square
	\begin{center}
		\begin{tikzpicture}[baseline=(current  bounding  box.south), scale=2]
			
			\node (a) at (0,0.7) {$S_ M  T_ M (C,x)$};
			\node (b) at (1.9,0.7) {$S_ M  T_ M (D,y)$};
			\node (c) at (0, 0) {$(C,x)$};
			\node (d) at (1.9, 0) {$(D,y)$};

			\path[font=\scriptsize]
			
			(c) edge [->] node [left] {$\eta_{(C,x)} $} (a)
			(d) edge [->] node [right] {$\eta_{(D,y)} $} (b)
			(a) edge [->] node [above] {$S_ M  T_ M (f_i)$} (b)
			(c) edge [->] node [below] {$f_i$} (d);
			
		\end{tikzpicture}	
	\end{center} 
	It's then enough to find a map out of $S_ M  T(D,y)$ which coequalizes the $S_ M  T(f_i)$'s. Therefore, without loss of generality, we can assume $f_i=S_ M (g_i)$ for some $g_i$ in $\C_0$. Now we can argue as in the previous case: $y$ defines an arrow $\bar{y}\colon I \to \tx{Eq}\ U( M (g_i))_i$ in $\V$ and we have
	\begin{equation*}
		\begin{split}
			\tx{Eq}\ U( M (g_i))_i &\cong U\left( \tx{Eq}( M * \C(g_i,-))_i\right)\\
			&\cong U\left( M * \tx{Eq}\ \C(g_i,-)_i\right)\\
			&\cong  M _U* \tx{Eq}\ (U_*\C)(f_i,-)_i\\
		\end{split}
	\end{equation*}
	As a consequence, $\bar{y}$ corresponds to a map $y'\colon I\to  M _U* \tx{Eq}(U_*\C)(f_i,-)_i$. By Lemma~\ref{multigoounit} there exist then $E\in U_*\C$, $z\colon I\to  M _U (E)$ and $f\in \tx{Eq}\ (U_*\C)_0(f_i,E)$ which map down to $x'$. In other words we obtained an object $(E,z)$ and a map $g\colon (D,y)\to (E,z)$ in $\tx{El}( M _{LI})$ coequalizing the $f_i$'s. It follows that $\tx{El}( M _{LI})$ is $\alpha$-filtered.
	
	$(3)\Rightarrow (1)$. Assume now that $ M _U$ is $\alpha$-flat; it's enough to prove that the $\W$-functor $ M *-\colon [\C,\W]\to \W$ preserves all $\alpha$-small conical limits (by Remark~\ref{powerswhocares}). Since $U$ is continuous and conservative, that holds if and only if  $U( M *-)\colon [\C,\W]_0\to \V_0$ preserves all $\alpha$-small conical limits. But $U( M *-)\cong  M _U*(-)_U$, where $(-)_U$ is continuous (by Lemma~\ref{underlyingGfunct}) and $ M _U*-$ preserves $\alpha$-small conical limits because $ M _U$ is $\alpha$-flat. Thus the result follows.
\end{proof}

The following is a consequence of the results above and is what justifies the notion of {\em protofiltered} colimit introduced below.

\begin{cor}\label{protoelements}
	Let $\C$ be a $\W$-category with finite direct sums and copowers by $L\G$, and let $ M \colon \C^{op}\to\W$ be an $\alpha$-flat $\W$-weight. Then:\begin{itemize}
		\item[(i)] $\tx{El}( M _{LI})$ is $\alpha$-filtered;
		\item[(ii)] $S_ M \colon \tx{El}( M_J )\to\tx{El}( M _{LI})$ is final.
	\end{itemize}
\end{cor}
\begin{proof}
	The first point follows directly by the proof of $(1)\Rightarrow (3)$ above. The second point holds since $S_ M $ has a left adjoint by Proposition~\ref{underlyingleftadj}.
\end{proof}

Given a $\W$-category $\A$ consider the ordinary functor $S_\A\colon \A_0\to( U_*\A)_0$ introduced near the beginning of Section~\ref{DG-setting}.

\begin{Def}
	We say that an ordinary functor $S\colon \E\to\F$ is an {\em $\alpha$-protofiltered index} if $\F$ is $\alpha$-filtered and $S$ is final. \\
	An {\em $S$-indexed diagram} in a $\W$-category $\A$ is a pair of functors $(H_1,H_2)$ making the diagram
	\begin{center}
		\begin{tikzpicture}[baseline=(current  bounding  box.south), scale=2]
			
			\node (a0) at (0,0.7) {$\F$};
			\node (b0) at (1,0.7) {$( U_*\A)_0$};
			\node (c0) at (0,0) {$\E$};
			\node (d0) at (1,0) {$\A_0$};
			
			\path[font=\scriptsize]
			
			(a0) edge [dashed, ->] node [above] {$H_2$} (b0)
			(c0) edge [->] node [left] {$S$} (a0)
			(d0) edge [->] node [right] {$S_\A$} (b0)
			(c0) edge [dashed, ->] node [below] {$H_1$} (d0);
			
		\end{tikzpicture}	
	\end{center}
	commute (strictly). We define its colimit, if it exists, as $\colim(H_1,H_2):=\colim H_1$ the conical colimit of $H_1$ in $\A$.
\end{Def}

When $\colim(H_1,H_2)$ exists in $\A$ then $\colim(S_\A H_1)$ exists as a conical colimit in $ U_*\A$ and 
$$S_\A(\colim H_1)\cong \colim (S_\A H_1)\cong \colim H_2$$ 
where the first isomorphism holds because $S_\A$ preserves all conical colimits that exist in $\A$ (by Lemma~\ref{U-copowers}), while the latter is true since $S$ is final.

\begin{obs}
	Note that if $S$ is an $\alpha$-protofiltered index the category $\C$ is in general not $\alpha$-filtered (the protosplit coequalizer below gives a counterexample); see also \ref{final-filtered-subcat}.
\end{obs}

\begin{es}\label{example-DG-index}$ $
	\begin{itemize}\setlength\itemsep{0.25em}
		\item {\bf Protosplit coequalizers}. The {\em protosplit index} is defined as the inclusion  $H$ of the free-living pair into the free-living split pair. Then protosplit coequalizers are just $H$-indexed colimits; these are $\alpha$-filtered indexes for any $\alpha$. For $\W=\bo{DGAb}$ protosplit coequalizers were first introduced in \cite{NST2020cauchy}.
		\item {\bf Categories of elements}. Let $ M \colon \C^{op}\to\W$ be as in Proposition~\ref{DG-G-flat} above; then the functor $S_ M \colon \tx{El}( M _J)\to\tx{El}( M _{LI})$ is an $\alpha$-protofiltered index. 
	\end{itemize}
\end{es}

There is a way to express $\alpha$-protofiltered colimits as honest $\alpha$-flat weighted colimits: see Section~\ref{appendix}.

The following is needed for the characterization of $\alpha$-flat $\W$-functors.

\begin{prop}
	Let $\C$ be a $\W$-category with finite direct sums and copowers by $L\G$, and let $ M \colon \C^{op}\to\W$ be an $\alpha$-flat $\W$-weight; then 
	$$  M \cong\tx{colim} \left(\tx{El}( M _J)_{\W} \stackrel{\pi_{\W}}{\longrightarrow} \C \stackrel{Y}{\longrightarrow} [\C^{op},\W]\right). $$
\end{prop}
\begin{proof}
	As usual it's enough to prove that for any $C\in\C$ the comparison map between
	$  M (C)$ and the colimit of $\C(C,\pi-)_0\colon \tx{El}( M _J)\to  \W_0$ is an isomorphism.
	For that, note that we can consider the commutative square below.
	\begin{center}
		\begin{tikzpicture}[baseline=(current  bounding  box.south), scale=2]
			
			\node (a0) at (0,0.7) {$\tx{El}( M _{LI})$};
			\node (b0) at (1.1,0.7) {$(U_*\C)_0$};
			\node (c0) at (0,0) {$\tx{El}( M _J)$};
			\node (d0) at (1.1,0) {$\C_0$};
			\node (e0) at (2.5,0) {$\W_0$};
			\node (f0) at (2.5,0.7) {$\V_0$};
			
			\path[font=\scriptsize]
			
			(a0) edge [->] node [above] {$\pi$} (b0)
			(c0) edge [->] node [left] {$S_ M $} (a0)
			(d0) edge [->] node [right] {$S_\C$} (b0)
			(c0) edge [->] node [below] {$\pi'$} (d0)
			(e0) edge [->] node [right] {$U$} (f0)
			(d0) edge [->] node [below] {$\C(-,C)_0$} (e0)
			(b0) edge [->] node [above] {$(U_*\C)(C,-)_0$} (f0);
			
		\end{tikzpicture}	
	\end{center}
	Since $S_ M $ is final (Corollary~\ref{protoelements}) and $U$ is conservative and cocontinuous, it's then enough to show that $ M _U(C)=U M (C)$ is the colimit of $(U_*\C)(C,\pi'-)_0$. But this is a consequence of Proposition~\ref{flat+absolute=filtered} since $ M _U$ is $\alpha$-flat by Proposition~\ref{DG-G-flat}.
\end{proof}

Then, under the presence of some absolute colimits, we can express every $\alpha$-flat colimit as an $\alpha$-protofiltered colimit of representables:

\begin{cor}\label{DG-flat-proto}
	Let $\C$ be a $\W$-category with finite direct sums and copowers by $L\G$, and let $ M \colon \C^{op}\to\W$ be an $\alpha$-flat $\W$-weight. Then $ M $ is an $\alpha$-protofiltered colimit of representables; more precisely $ M $ is the $\alpha$-protofiltered colimit of the diagram below.
	\begin{center}
		\begin{tikzpicture}[baseline=(current  bounding  box.south), scale=2]
			
			\node (a0) at (0,0.7) {$\tx{El}( M _{LI})$};
			\node (b0) at (1.1,0.7) {$(U_*\C)_0$};
			\node (c0) at (0,0) {$\tx{El}( M _J)$};
			\node (d0) at (1.1,0) {$\C_0$};
			\node (e0) at (2.5,0) {$[\C^{op},\W]_0$};
			\node (f0) at (2.5,0.7) {$(U_*[\C^{op},\W])_0$};
			
			\path[font=\scriptsize]
			
			(a0) edge [->] node [above] {$\pi$} (b0)
			(c0) edge [->] node [left] {$S_ M $} (a0)
			(d0) edge [->] node [right] {$S_\C$} (b0)
			(c0) edge [->] node [below] {$\pi$} (d0)
			(e0) edge [->] node [right] {$R_{[\C,\W]}$} (f0)
			(d0) edge [->] node [below] {$Y_0$} (e0)
			(b0) edge [->] node [above] {$(U_*Y)_0$} (f0);
			
		\end{tikzpicture}	
	\end{center}
\end{cor}
\begin{proof}
	This is a direct consequence of the proposition above and Example~\ref{example-DG-index}.
\end{proof}

In conclusion:

\begin{teo}\label{flat-dgab}
	If $ M \colon \C^{op}\to\W$ is a $\W$-functor, the following are equivalent:\begin{enumerate}\setlength\itemsep{0.25em}
		\item $ M $ is $\alpha$-flat;
		\item $ M *-\colon [\C,\W]\to \W$ preserves $\alpha$-small weighted limits of representables;
		\item $ M $ lies in the closure of the representables under copowers by $L\G$, finite direct sums, and $\alpha$-protofiltered colimits.
	\end{enumerate}
\end{teo}
\begin{proof}
	Same as that of Proposition~\ref{flat-ab}.
\end{proof}

\begin{teo}\label{flat=protofilt+abs}
	A $\W$-category $\A$ has $\alpha$-flat colimits if and only if it has finite direct sums, copowers by $L\G$, and $\alpha$-protofiltered colimits. A $\V$-functor from such an $\A$ preserves $\alpha$-flat colimits if and only if it preserves $\alpha$-protofiltered colimits.
\end{teo}
\begin{proof}
	Same as that of Theorem~\ref{con-abs=flat}.
\end{proof}

Since in this case we don't know whether the flat colimits are generated by absolute and filtered colimits, we can't compare accessible and conically accessible $\V$-categories (as we did in the previous sections). What we can say is the following:

\begin{teo}
	Let $\A$ be a $\W$-category with $\alpha$-flat colimits; then:\begin{enumerate}
		\item $A\in\A_\alpha$ if and only if $\A(A,-)$ preserves $\alpha$-protofiltered colimits;
		\item $\A$ is $\alpha$-accessible if and only if $\A_\alpha$ is small and every object is an $\alpha$-protofiltered colimit of $\alpha$-presentable objects.
	\end{enumerate}
\end{teo}
\begin{proof}
	$(1)$ is a consequence of the Theorem~above. To prove $(2)$ we use Theorem~\ref{flat-dgab} and argue as in the proof of Theorem~\ref{acc=conacc+cauchy}. Given an $\alpha$-flat $ M \colon \C^{op}\to\W$ and a diagram $K\colon \C\to\A_{\alpha}\subseteq\A$,
	we can consider the free cocompletion $\D$ of $\C$ under finite direct sums and $L\G$-copowers, with inclusion $J\colon\C\to\D$. Let $ M ':=\tx{Lan}_{J^{op}} M $ and $K':=\tx{Lan}_JK$; then: $M*K\cong M'*K'$, the diagram $K'$ still lands in $\A_\alpha$ (since this is closed in $\A$ under finite direct sums and $L\G$-copowers), the weight $M'$ is still $\alpha$-flat (Lemma~\ref{flat-restriction}), and the domain of $M'$ satisfies the hypotheses of Corollary~\ref{DG-flat-proto}. Thus we can write $ M '\cong\tx{colim}(Y_0H_1,(U_*Y)_0H_2)$ as an $\alpha$-protofiltered colimit of representables, where $(H_1,H_2)$ is an $\alpha$-protofiltered diagram in $\D$. It follows that 
	$$ M *K\cong \tx{colim}\ (K_0'H_1,(U_*K')_0H_2),$$
	either side existing if the other does. Thus an object of $\A$ is an $\alpha$-flat colimit of $\alpha$-presentables if and only if it is an $\alpha$-protofiltered colimit of $\alpha$-presentables. The result then follows.
\end{proof}

\begin{obs}
	Let $\A$ be an $\alpha$-accessible $\W$-category; then every object $A$ of $\A$ can be written canonically as the $\alpha$-protofiltered colimit of the diagram below.
	\begin{center}
		\begin{tikzpicture}[baseline=(current  bounding  box.south), scale=2]
			
			\node (a0) at (0,0.7) {$(U_*\A_\alpha)_0/A$};
			\node (b0) at (1.5,0.7) {$(U_*\A)_0$};
			\node (c0) at (0,0) {$(\A_\alpha)_0/A$};
			\node (d0) at (1.5,0) {$\A_0$};
			
			\path[font=\scriptsize]
			
			(a0) edge [->] node [above] {$\pi$} (b0)
			(c0) edge [->] node [left] {$S_A$} (a0)
			(d0) edge [->] node [right] {$S_\A$} (b0)
			(c0) edge [->] node [below] {$\pi$} (d0);
			
		\end{tikzpicture}	
	\end{center}

	Indeed, let $H\colon\A_\alpha\to\A$ be the inclusion; since $\A$ is $\alpha$-accessible, every object $A$ can be written as the $\alpha$-flat colimit $\A(H-,A)*H$. Then the result follows thanks to Corollary~\ref{DG-flat-proto} applied to $M=\A(H-,A)$ since $(\A_\alpha)_0/A=\tx{El}(\A(H-,A)_J)$ and $(U_*\A_\alpha)_0/A= \tx{El}(\A(H-,A)_{LI})$.
\end{obs}

We say that an index $J\colon \E\to\F$ is {\em protoabsolute} if it is $\alpha$-protofiltered for every $\alpha$; equivalently if $\F$-colimits are Cauchy in the ordinary sense. The following result then generalizes \cite[Theorem~7.2]{NST2020cauchy} to our setting:

\begin{cor}\label{DG-Cauchy}
	Let $\C$ be a $\W$-category; the following are equivalent:
	\begin{enumerate}\setlength\itemsep{0.25em}
		\item $\C$ is Cauchy complete;
		\item $\C$ has finite direct sums, powers and copowers by dualizable objects, and protoabsolute colimits;
		\item $\C$ has finite direct sums, copowers by $L\G$, and protoabsolute colimits.
	\end{enumerate}
\end{cor}
\begin{proof}
	The proof is completely analogous to that of Corollary~\ref{Cauchy}.
\end{proof}

\subsection{Protofiltered colimits as weighted colimits}\label{appendix}

Let $\W,\V$, and $U\colon\W\to\V$ be as in Section~\ref{DG-setting}. Given an $\alpha$-protofiltered index $S\colon\E\to\F$ we construct a $\W$-category $\S $ and a weight $\Delta\colon\S ^{op}\to\W$ such that $S$-indexed colimits correspond to $\Delta$-weighted colimits and $\Delta$ is $\alpha$-flat.

\begin{obs}
	The 2-category $\W\tx{-}\bo{Cat}$ is locally finitely presentable as a 2-category. Indeed its underlying ordinary category $\W\tx{-}\bo{Cat}_0$ is locally finitely presentable by \cite[Theorem~4.5]{KL2001:articolo} and the finitely presentable objects are closed under copowers by $\mathbbm{2}$ \cite[Proposition~4.8]{KL2001:articolo}. It follows that every  object which is finitely presentable in the ordinary sense is also finitely presentable in the 2-categorical sense. Therefore $\W\tx{-}\bo{Cat}$ is locally finitely presentable by \cite[7.5]{Kel82:articolo}.
\end{obs}

Let us first define $\S $. Consider the 2-functor $S_*\colon\W\tx{-}\bo{Cat}\to \bo{Cat}$ defined pointwise as the pullback
\begin{center}
	\begin{tikzpicture}[baseline=(current  bounding  box.south), scale=2]
		
		\node (a0) at (0,0.8) {$S_*\A$};
		\node (b0) at (1.8,0.8) {$\bo{Cat}(\F,(U_*\A)_0)$};
		\node (c0) at (0,0) {$\bo{Cat}(\E,\A_0)$};
		\node (d0) at (1.8,0) {$\bo{Cat}(\E,(U_*\A)_0)$};
		\node (e0) at (0.2,0.6) {$\lrcorner$};
		
		\path[font=\scriptsize]
		
		(a0) edge [->] node [above] {} (b0)
		(a0) edge [->] node [left] {} (c0)
		(b0) edge [->] node [right] {$-\circ S$} (d0)
		(c0) edge [->] node [below] {$S_\A\circ -$} (d0);
	\end{tikzpicture}	
\end{center}
in $\bo{Cat}$ for any $\A\in\W\tx{-}\bo{Cat}$, where $S_\A\colon \A_0\to( U_*\A)_0$ is the ordinary functor introduced in Section~\ref{DG}.
Note that $U_*\colon\W\tx{-}\bo{Cat}\to \V\tx{-}\bo{Cat},\  (-)_0\colon\W\tx{-}\bo{Cat}\to\bo{Cat},$ and $(-)_0\colon\V\tx{-}\bo{Cat}\to\bo{Cat}$ are continuous as 2-functors; indeed they preserves all conical limits as well as powers by $\mathbbm{2}$ (since $U$ is continuous and strong closed). Moreover they preserve $\alpha$-filtered colimits for some $\alpha$. Similarly $\bo{Cat}(\E,-)$ and $\bo{Cat}(\F,-)$ are continuous and preserve $\beta$-filtered colimits for some $\beta\geq \alpha$ (as is true for every object of a locally presentable category).
It follows that $S_*$ is a finite limit of 2-functors which are continuous and preserve $\beta$-filtered colimits. Since finite limits commute with $\beta$-filtered colimits in $\bo{Cat}$, it follows that $S_*$ is continuous and preserves $\beta$-filtered colimits as well. As a consequence, since $\W\tx{-}\bo{Cat}$ and $\bo{Cat}$ are locally presentable 2-categories, it follows that $S_*$ has a left adjoint \cite[7.9]{Kel82:articolo} and therefore $S_*$ is a representable 2-functor. Now we can define $\S $ as the $\W$-category which represents $S_*$.

For any {\em small} $\W$-category $\A$ we have an isomorphism of categories $[\S ,\A]_0\cong S_*\A$; it's easy to see that this isomorphism actually holds for any (possibly large) $\W$-category $\A$. It follows in particular that to give a $\W$-functor $H\colon\S \to\A$ is the same as to give a $S$-indexed diagram $(H_1,H_2)$ in $\A$. The same holds for $\W$-functors out of $\S ^{op}$ just by moving the $(-)^{op}$ to the codomain.  

\begin{obs}\label{K_1K_2}
	Taking $\A=\S $, we note that the identity on $\S $ corresponds to functors $K_1\colon\E\to(\S )_0$ and $K_2\colon\F\to (U_*\S )_0$. These two functors induce the bijection just mentioned: given $H\colon\S \to\A$ then $H_1= H_0\circ K_1$ and $H_2= (U_*H)_0\circ K_2$.
	
	In particular it follows that $H_1^t= H\circ K_1^t$ and $H_2^t= U_*H\circ K_2^t$, where $K_1^t\colon\E_\W\to\S$ and $K_2^t\colon\F_\V\to U_*\S$ are the transposes of $K_1$ and $K_2$ respectively. Therefore for any $A\in \A$ we obtain $\A(H-,A)_1^t\cong\A(H_1^t-,A)$, and for any $X\in\V$ we have $X\pitchfork H_1^t\cong (X\pitchfork H)_1^t$ whenever such powers exist in $\A$.
\end{obs}

Next we define the weight $\Delta\colon\S ^{op}\to\W$ as the one such that 
$\Delta^{op}$ corresponds under the isomorphism $\W\tx{-}\bo{Cat}(\S ,\W^{op})\cong S_*\W^{op}$ to the $S$-indexed diagram $$(\Delta J\colon\E\to\W^{op}_0,\Delta J\colon \F\to (U_*\W)^{op}_0).$$

We can now prove that limits and colimits weighted by $\Delta$ are the same as $S$-indexed limits and colimits.

\begin{prop}
	Let $H\colon\S ^{op}\to\W$ be any $\W$-functor with induced $S$-index $(H_1,H_2)$ in $\W$. Then
	$$ [\S ^{op},\W](\Delta,H)\cong [\E_{\W}^{op},\W](\Delta J^t,H_1^t) $$
	where $\Delta J^t,H_1^t\colon\E_{\W}^{op}\to\W$ are the transposes of $\Delta J,H_1\colon\E^{op}\to\W_0$, and the isomorphism in induced by precomposition with $K_1^t$. In other words $\{\Delta,H\}\cong \lim H_1$.
\end{prop} 
\begin{proof}
	The isomorphism above holds if and only if applying $\W_0(X,-)$ on each side we get a bijection of sets. But 
	$$\W_0(X,[\S ^{op},\W](\Delta,H))\cong [\S ^{op},\W]_0(\Delta,X\pitchfork H),$$ 
	similarly on the right-hand-side (using that $X\pitchfork H_1^t\cong (X\pitchfork H)_1^t$, see Remark~\ref{K_1K_2}). Therefore it's enough to prove that we have a bijection between $\W$-natural transformations $\eta\colon\Delta\Rightarrow H$ and $\eta_1\colon\Delta J\Rightarrow H_1$ for any $H$.
	
	Since we have an isomorphism of categories $[\S ,\W^{op}]_0\cong S_*\W^{op}$, it follows that to give a $\W$-natural transformation $\eta\colon\Delta\Rightarrow H$ is the same as giving a pair of ordinary natural transformations $\eta_i\colon\Delta J\Rightarrow H_i$, for $i=1,2$, such that $S_\W\eta_1= \eta_2 S$. Since $S$ is final, $\eta_2$ is uniquely determined by $S_\W\eta_1$; therefore it's enough to give $\eta_1$ and hence we have the desired bijection.
\end{proof}

An immediate consequence is the following:

\begin{cor}
	Let $\A$ be any $\W$-category. For any $H\colon\S \to\A$ let $(H_1,H_2)$ be the induced $S$-indexed diagram in $\A$; then 
	$$ \Delta* H\cong \colim(H_1,H_2) $$
	either side existing if the other does. 
\end{cor}
\begin{proof}
	It's enough to consider the following chain of isomorphisms for any $A\in\A$:
	\begin{equation*}
		\begin{split}
			\A(\Delta*H,A) &\cong [\S^{op},\W](\Delta, \A(H-,A))\\
			&\cong [\E_\W^{op},\W](\Delta J^t, \A(H-,A)_1^t)\\
			&\cong [\E_\W^{op},\W](\Delta J^t, \A(H_1^t-,A))\\
			&\cong \A(\colim H_1,A)\\
		\end{split}
	\end{equation*}
	where the second isomorphism holds by the proposition above, while the third follows from the fact that $\A(H-,A)_1^t\cong\A(H_1^t-,A)$ from Remark~\ref{K_1K_2}.
\end{proof}

The final step is to prove that the weight $\Delta$ is $\alpha$-flat.

\begin{prop}
	Let $S\colon\E\to\F$ be an $\alpha$-protofiltered index; then $\Delta\colon\S ^{op}\to\W$ is an $\alpha$-flat $\W$-weight.
\end{prop}
\begin{proof}
	First note that there is a commutative triangle as below,
	\begin{center}
		\begin{tikzpicture}[baseline=(current  bounding  box.south), scale=2]
			
			\node (a) at (0.7,-0.5) {$(U_*\W)_0$};
			\node (c) at (-0.1, 0) {$\W_0$};
			\node (d) at (1.5, 0) {$\V_0$};
			
			\path[font=\scriptsize]
			
			(c) edge [->] node [below] {$S_\W\ \ \ \ \ \ $} (a)
			(a) edge [->] node [below] {$\ \ \ \ \ \ \hat{U}_0$} (d)
			(c) edge [->] node [above] {$U$} (d);
			
		\end{tikzpicture}	
	\end{center} 
	where $\hat{U}=(U_*\W)(I,-)\colon U_*\W\to\V$ was introduced in Section~\ref{DG-setting} and is such that $(-)_U= \hat{U}\circ U_*(-)$.
	Let $H\colon\S \to\W$ be a $\W$-functor; then by the corollary above we know that
	$$ U(\Delta*H)\cong U(\colim_\E H_1)\cong \colim_\E (\hat{U}_0S_\W H_1)\cong \colim_\E(\hat{U}_0H_2S)\cong \colim_\F (\hat{U}_0H_2) $$
	and this can be rewritten as
	$$ U(\Delta*-)\cong \colim_\F ( (-)_U\circ K_2^t )\colon [\S ,\W]_0\to \V_0$$
	where $K_2^t\colon\F_\V\to U_*\E_J$ is the transpose of the ordinary functor $K_2\colon\F\to (U_*\S )_0$ introduced in Remark~\ref{K_1K_2}. Now, $(-)_U$ and $(-\circ K_2^t)$ are continuous and $\colim_\F(-)$ preserves $\alpha$-small ordinary limits ($\F$ is $\alpha$-filtered); therefore $U(\Delta*-)$ preserves $\alpha$-small limits as well. To conclude then note that, since $U$ is continuous and conservative, $\Delta*-$ preserves all $\alpha$-small conical limits, and this is enough to guarantee that $\Delta$ is $\alpha$-flat by Remark~\ref{powerswhocares}.  
\end{proof}


\end{document}